\documentclass[10pt]{article}
\usepackage{amsmath, amsthm, amssymb, amsfonts, mathrsfs, mathtools, upgreek}
\usepackage{graphicx}
\usepackage{makeidx}
\usepackage{tikz-network}
\usepackage[left=2cm,right=2cm,top=2cm,bottom=2cm]{geometry}
\usepackage{caption}
\usepackage{subcaption}
\usepackage[section]{placeins}
\usepackage{enumitem}
\usepackage{tikz}
\usepackage{stmaryrd}
\usepackage{hyperref}
\usepackage{upquote}

\usepackage{hyperref}
\hypersetup{colorlinks=true, citecolor={blue}, linkcolor=blue, filecolor=magenta, urlcolor=cyan}

\newtheorem{theorem}{Theorem}[section]

\newtheorem{lemma}[theorem]{Lemma}
\newtheorem{coro}[theorem]{Corollary}
\newtheorem{defi}[theorem]{Definition}
\newtheorem{ex}[theorem]{Example}

\numberwithin{equation}{section}
\newcommand{\ip}[2]{\left\langle #1, #2 \right\rangle}

\newcommand{\aut}{\operatorname{Aut}}

\newcommand{\cay}{\operatorname{Cay}}

\newcommand{\jb}{\mathbf{j}}

\def \r {{\rvert}}

\def \Fl {{\mathbb F}}
\def \Zl {{\mathbb Z}}
\def \Nl {{\mathbb N}}
\def \Rl {{\mathbb R}}
\def \Zl {{\mathbb Z}}
\def \Ql {{\mathbb Q}}
\def \Cl {{\mathbb C}}

\def \ld {{\lambda}}
\def \e {{\mathbf{e}}}
\def \eu {{\mathbf{e}_u}}
\def \ev {{\mathbf{e}_v}}
\def \evt  {{\mathbf{e}^t_v}}
\def \eut  {{\mathbf{e}^t_u}}

\title{\textbf{Pretty good state transfer in Grover walks on abelian Cayley graphs}}

\author{Koushik Bhakta\thanks{Corresponding author.} \ and Bikash Bhattacharjya\\
	Department of Mathematics\\
	Indian Institute of Technology Guwahati, India\\
	b.koushik@iitg.ac.in, b.bikash@iitg.ac.in }

\date{}
\begin{document}
	\maketitle
	
	\vspace{-0.3in}
	
	\begin{center}{\textbf{Abstract}}\end{center}
	\noindent 
	In this paper, we study pretty good state transfer (PGST) between vertex states in Grover walks on graphs. Using Chebyshev polynomials, we obtain a spectral and number-theoretic characterization for the occurrence of PGST on general graphs and on Cayley graphs over finite abelian groups. This approach yields complete classifications of hypercubes, the prisms $K_m\square K_2$, and unitary Cayley graphs. In particular, the hypercube $Q_d$ exhibits PGST if and only if $d$ is odd, $d\equiv2\pmod4$, or $d\equiv8\pmod{16}$, extending the known result for only prime $d$ (Chan and Zhan~\cite{pgst1}). We show that the independent-set blow-up $G[\overline K_2]$ exhibits PGST if discriminant eigenvalues of $G$ are reciprocals of nonzero integers. Using techniques from algebraic number theory, we obtain several infinite families of graphs exhibiting PGST, including examples with rationally dependent spectral angles not covered by earlier constructions based on linear independence over $\mathbb Q$.

	\vspace*{0.3cm}
	
	\section{Introduction}
	Quantum walks are the quantum analogues of classical random walks~\cite{random}. They provide useful tools for quantum algorithms and quantum information processing~\cite{algorithm,information}. One basic problem is to determine whether a quantum state localized at one vertex of a graph can be transferred to another vertex. Perfect state transfer (PST) requires exact transfer, whereas pretty good state transfer (PGST) requires the target state to be approached with arbitrarily small error. Since the spectral conditions for PST are restrictive, PGST allows us to study transfer on graphs where exact transfer is not possible. It also raises a natural question: which properties of a graph determine whether such approximation is possible? For continuous-time quantum walks, this question has received considerable attention; see~\cite{pgst_paths,pgst_dihedral,pgst_double_star,state,pal_pgst2} and the references therein.
	
	In this paper, we consider the Grover walk, a discrete-time quantum walk whose evolution is determined by a unitary matrix acting on the arcs of a graph~\cite{godsildct,spec}. We study transfer between vertex states, in which the amplitudes are uniformly distributed over the arcs directed towards a given vertex. The Grover walk acts on an arc space, but much of its behaviour is determined by the discriminant matrix, a real symmetric matrix indexed by the vertices. Thus, state transfer in Grover walks is closely related to graph spectra, spectral projections, and polynomial functions of matrices. These connections have been used to study PST~\cite{bhakta4,kubota1,kubota2,zhan1} and periodicity~\cite{ito,mixedpaths,odd-grover,distance}.
	
	Chan and Zhan~\cite{pgst1} initiated the general study of PGST in discrete-time quantum walks. They gave spectral conditions and several infinite families of examples. In particular, they proved PGST for cocktail party graphs and for hypercubes of prime dimension. Their general criterion allows rational relations among the spectral angles, while their construction for prime-dimensional hypercubes uses linear independence over $\mathbb Q$ of suitable angles $\arccos\mu$, together with $\pi$. Here $\mu$ denotes a discriminant eigenvalue. They left open the determination of all dimensions for which a hypercube exhibits PGST in the Grover walk. Subsequently, Zhan~\cite{pgst2} obtained PGST on every hypercube by choosing suitable weighted Grover coins. The classification for the standard Grover coin remains a separate question, since changing the coin changes the evolution of the walk.
	
	We use Chebyshev polynomials of the first kind to study PGST directly through the discriminant matrix and obtain explicit classifications for the standard Grover walk. To the best of our knowledge, this is the first systematic study of PGST in Grover walks using Chebyshev polynomials. The starting point is the identity $NU^tN^*=T_t(P)$,	established by Kubota and Segawa~\cite{kubota1} in their study of PST, where $U$ is the time evolution matrix, $N$ is the boundary matrix, and $P$ is the discriminant matrix. Guo and Schmeits~\cite{guo_sch} also used projected transition matrices and Chebyshev polynomials to study peak state transfer. For PGST, we study limits of $T_t(P)$ and the simultaneous approximation of $T_t(\mu)$ to prescribed signs. Together with Kronecker's approximation theorem, this gives a direct proof of a spectral criterion for PGST and a convenient form for applications to abelian Cayley graphs.
	
	We first give spectral and number-theoretic characterizations of PGST on general graphs (Theorems~\ref{ch_pgst2} and~\ref{pgst_graphs}). To construct graph families exhibiting PGST, we consider independent-set blow-ups, obtained by replacing each vertex by an independent set and each edge by all edges between the corresponding sets. State transfer on these graphs has been studied for continuous-time quantum walks~\cite{bmp_blowups}, and peak state transfer has been studied for Grover walks~\cite{guo_sch}. For PGST, this construction leads to simultaneous approximation problems involving angles of the form $\arccos(1/a)$. Using results on geodetic angles, Pell equations, and primitive prime divisors, we prove that $\pi$ and $\arccos(1/a_i)$ are linearly independent over $\mathbb Q$ for distinct integers $a_i\geq3$ (Theorem~\ref{li_arccos_i}). Consequently, $G[\overline K_2]$ exhibits PGST whenever every discriminant eigenvalue of $G$ is the reciprocal of a nonzero integer (Theorem~\ref{reciprocal_blowup}). This extends the cocktail party construction of Chan and Zhan \cite{pgst1} to blow-ups of tensor products of complete graphs (Corollary~\ref{complete_tensor_blowup}).

	For abelian Cayley graphs, we express the criterion in terms of their characters (Theorem~\ref{pgst_cayley}). These methods also give PGST on a cubelike graph of degree $d$ when the sum of its connection set is nonzero and $v_2(d)\in\{0,1,3\}$, where $v_2(d)$ is the exponent of $2$ in $d$ (Theorem~\ref{cont_cubelike}). We show that the prism $K_m\square K_2$ exhibits PGST exactly when $m\ne4$, for $m\geq2$ (Theorem~\ref{prism_independent}), and completely classify the unitary Cayley graphs that exhibit PGST (Theorem~\ref{pgst_unitary}).
	
	Chan and Zhan~\cite[Theorem~7.7]{pgst1} proved PGST on hypercubes of odd prime dimension using linear independence over $\mathbb Q$ of suitable spectral angles together with $\pi$. We give a complete classification: $Q_d$ exhibits PGST precisely when $d$ is odd, $d\equiv2\pmod4$, or $d\equiv8\pmod{16}$, and transfer occurs only between antipodal vertices (Theorem~\ref{cont_hypercube}). This includes dimensions where the independence used in their proof fails. For example, $Q_9$ exhibits PGST, although its discriminant eigenvalues $1/3$ and $7/9$ give the relation $2\arccos(1/3)+\arccos(7/9)=\pi$, which follows from $T_2(1/3)=-7/9$. Our proof uses algebraic number theory to verify the required parity conditions on angle relations without assuming linear independence.

	The remaining part of the paper is organized as follows. Section~2 gives the necessary background on Grover walks and PGST. Section~3 develops the general characterization of PGST, proves the result on reciprocal-cosine angles, and studies independent-set blow-ups. Section~4 gives the character criterion for abelian Cayley graphs and applies it to cubelike graphs, hypercubes, prisms, and unitary Cayley graphs.

	\section{Preliminaries}
	Let $G$ be a finite simple graph without isolated vertices, with the vertex set $V(G)$ and the edge set $E(G)$. We denote the elements of $E(G)$ by the unordered pairs $uv$ of distinct elements of $V(G)$. The ordered pairs $(u,v)$ and $(v,u)$ are known as the \emph{arcs} of $G$. The set of  symmetric arcs of $G$ is defined as $\mathcal{A}(G)=\{(u, v), (v, u):uv\in E(G) \}$. For an arc $a:=(u,v)$, the origin and terminus of $a$ are defined by $o(a)=u$ and $t(a)=v$, respectively. The inverse arc of $a$, denoted $a^{-1}$, is the arc $(v,u)$.  
	
	We define the following fundamental matrices associated with Grover walks on a graph $G$. We refer the reader to \cite{kubota2} for more details about these matrices. The \emph{shift matrix} $R:=R(G)\in \mathbb{C}^{\mathcal{A}(G) \times \mathcal{A}(G)}$ of $G$ is defined by $R_{ab}=\delta_{a,b^{-1}}$,  where $\delta_{a,b}$ is the Kronecker delta function. The shift matrix satisfies $R^2=I$, where $I$ is the identity matrix. The \emph{boundary matrix} $N:=N(G)\in \mathbb{C}^{V(G)\times \mathcal{A}(G)}$ of $G$ is defined by $N_{ua}=\frac{1}{\sqrt{\deg u}}\delta _{u, t(a)}$, where $\deg u$ is the degree of the vertex $u$. It follows that $NN^*=I$. The \emph{coin matrix} $C:=C(G)\in\mathbb{C}^{\mathcal{A}(G) \times \mathcal{A}(G)}$ of $G$ is defined by $C=2N^*N-I$. Finally, the \emph{time evolution matrix} $U:=U(G)\in \mathbb{C}^{\mathcal{A}(G)\times \mathcal{A}(G)}$ of $G$ is defined by $U=RC$. The discrete-time quantum walk defined by the unitary matrix $U$ is known as the \emph{Grover walk}.  The Grover walk is also recognized as a special case of the bipartite walk (see \cite{chen}).

	The spectral analysis of the time evolution matrix $U$ of a graph $G$ can be simplified by examining the spectrum of a reduced matrix, called the \emph{discriminant matrix} $P:=P(G)\in \mathbb{C}^{V(G)\times V(G)}$, defined as $P=NRN^*.$ The \emph{adjacency matrix} of a graph $ G $ is $A := A(G) \in \mathbb{C}^{V(G) \times V(G)}$, where the $uv$-th entry of $A$ is equal to 1 if $uv\in E(G)$, and 0 otherwise. In the case of regular graphs, the discriminant matrix is correlated to the adjacency matrix. 
	\begin{lemma}[{\cite[Theorem~3.1]{mixed1}}]
		Let $G$ be an $r$-regular graph. Then $P=\frac{1}{r}A$.
	\end{lemma}
	Kubota and Segawa \cite{kubota1} established a connection between the matrices $N$, $P$ and $U$ using Chebyshev polynomials. The \emph{Chebyshev polynomial of the first kind}, denoted $T_m(x)$, is the polynomial defined by the  recurrence relation:
	$T_0(x)=1,~ T_1(x)=x~ \text{and} ~T_m(x)=2xT_{m-1}(x)-T_{m-2}(x) ~\text{for}~ m \geq 2.$ It is well known that $T_m(\cos\theta)=\cos(m\theta)$, and so $|T_m(x)|\leq 1$ for $|x|\leq 1$.
	\begin{lemma}[{\cite[Lemma~3.1]{kubota1}}] \label{ch11}
		Let $T_m(x)$ be the Chebyshev polynomial of the first kind.	Let $G$ be a graph with the time evolution matrix $U$ and discriminant $P$. Then $NU^m N^* = T_m (P)$ for any positive integer $m$. 
	\end{lemma}
	Throughout the paper, we use the standard Euclidean inner product, with the norm defined by $\|x\|=\sqrt{\ip{x}{x}}$.
	\begin{lemma}[{\cite[Lemma~3.2]{kubota1}}]\label{sh12}
		Let $G$ be a graph with discriminant $P$. Then for $u\in V(G)$ and a positive integer $m$,  $\|T_m(P)\eu\|\leq1.$
	\end{lemma}
	A vector $\Phi\in \Cl^{\mathcal{A}(G)}$ is called a \emph{state} if $\|\Phi\|=1$. We say that perfect state transfer (PST) occurs from a state $\Phi$ to another state $\Psi$ at time $\tau\in \Nl$ if there exists a unimodular complex number $\gamma$ such that $U^\tau\Phi=\gamma \Psi$. The following characterization of PST is an immediate consequence of the Cauchy–Schwarz inequality.
	\begin{lemma}[{\cite[Section~4]{pgst1}}]\label{zhan_pst}
		A graph exhibits PST from a state $\Phi$ to another state $\Psi$ at time $\tau$ if and only if  $\left|\ip{U^\tau \Phi}{\Uppsi}\right|=1.$ 
	\end{lemma}
	We say that pretty good state transfer (PGST) occurs from a state $\Phi$ to another state $\Psi$ if there exists a unimodular complex number $\gamma$ such that, for any $\epsilon>0$, there exists a positive integer $\tau$ with  $\|U^\tau\Phi-\gamma\Psi\|<\epsilon$. Chan and Zhan \cite{pgst1} provided an equivalent characterization of PGST between two distinct states.
	\begin{lemma}[{\cite[Lemma~4.1]{pgst1}}]\label{zhan_pgst}
		Let $G$ be a graph and $U$ be its time evolution matrix. Then PGST occurs from a state $\Phi$ to another state $\Psi$  if and only if for any $\epsilon>0$, there exists a positive integer $\tau$ such that  $\left|\ip{U^\tau \Phi}{\Uppsi}\right|>1-\epsilon$. 
	\end{lemma}
	In this paper, we focus on transferring quantum states that are localized at the vertices of a graph $G$.  A state $\Phi\in \Cl^{\mathcal{A}(G)}$ is said to be \emph{vertex-type} if there exists  $u\in V(G)$ such that $\Phi=N^* \eu$, where $\eu$ is the unit vector corresponding to $u$, defined by $(\eu)_x=\delta_{u,x}$. A vertex-type state describes a quantum state with nonzero values only on the arcs directed towards a specific vertex. This corresponds to the physical scenario where a quantum walker is localized at a single vertex in a discrete quantum walk model. For a visual explanation and further motivation of vertex-type states, we refer the reader to \cite{kubota1}. Formally, the quantum state $\Phi_u$ associated with vertex $u$ is defined by $\Phi_u=N^* \eu$.
	\begin{defi}
		{\em  A graph exhibits \emph{PGST} from a vertex $u$ to another vertex $v$  if there exists a unimodular complex number $\gamma$ such that for any $\epsilon>0$, there exists a positive integer $\tau$ with 
			$\|U^\tau\Phi_u-\gamma\Phi_v\|<\epsilon.$}
	\end{defi}
	Similarly,  a graph exhibits \emph{PST} from a vertex $u$ to another vertex $v$ at time $\tau$ if there exists a unimodular complex number $\gamma$ such that $U^\tau\Phi_u = \gamma \Phi_v$.	We say that a graph exhibits PGST (respectively, PST) if there exist vertices $u$ and $v$ in the graph such that PGST (respectively, PST) occurs from $u$ to $v$. Throughout this paper, PGST is understood to occur between two distinct vertices.
	\section{Characterization for pretty good state transfer}
	
	\subsection{Spectral and number theoretic characterization}
	Since the discriminant matrix $P$ of a graph is symmetric, it has a spectral decomposition. Let $\sigma(P)$ denote the set of distinct eigenvalues of $P$. For $\mu\in\sigma(P)$, let $E_\mu$ denote the orthogonal projection of $P$ onto the eigenspace corresponding to $\mu$. Then the spectral decomposition of $P$ is 
	\[P=\sum_{\mu\in\sigma(P)}\mu E_\mu.\]
	The orthogonal projections $E_\mu$ satisfy the following relations: $E_\mu^2=E_\mu,~ E_\mu E_\eta=0~\text{for $\mu\neq\eta$},~ E_\mu^t=E_\mu$ and $\sum_{\mu\in\sigma(P)}E_\mu=I$. Using the preceding properties, if $f(x)$ is a polynomial, we have
	\begin{equation}\label{sd}
		f(P)=\sum_{\mu\in\sigma(P)} f(\mu) E_\mu.
	\end{equation}
	In the following result, we establish a connection between the Chebyshev polynomial and the occurrence of PGST on graphs.  
	\begin{lemma}\label{dpgst}
		Let $u$ and $v$ be two vertices of a graph $G$. Then the following are equivalent.
		\begin{enumerate}[label=(\roman*)]
			\item Pretty good state transfer occurs from $u$ to $v$ on $G$.
			\item There exists a sequence  of positive integers $\{\tau_k\}$ such that $\lim_{k\to \infty} T_{\tau_k}(P)\eu=\ev.$
			\item For any $\epsilon>0$, there exists a positive integer $\tau$ such that $|\evt T_\tau(P)\eu-1|<\epsilon.$
		\end{enumerate}		
	\end{lemma}
	\begin{proof}
		Let PGST occur from $u$ to $v$ on $G$. Then there exists a unimodular complex number $\gamma$ and a sequence of positive integers $\{\tau_k\}$ such that $\lim_{k\to \infty}U^{\tau_k}N^*\eu=\gamma N^*\ev$. Hence $\lim_{k\to \infty}NU^{\tau_k}N^*\eu=\gamma \ev$.
		Therefore by Lemma \ref{ch11}, we conclude that
		\begin{equation}\label{pgst_gamma=1}
			\lim_{k\to \infty} T_{\tau_k}(P)\eu=\gamma\ev.
		\end{equation}
		By \eqref{pgst_gamma=1}, $\gamma\in\{-1,1\}$. We now prove that $\gamma$ must be $1$. By \eqref{sd}, we have $T_m(P)=\sum_{\mu\in\sigma(P)} T_m(\mu) E_\mu$ for any positive integer $m$.  Note that $1$ is an eigenvalue of $P$ with corresponding eigenvector $D^\frac{1}{2}\mathbf{j}$, where $D$ is the diagonal matrix with $\eut D\eu=\deg u$ and $\jb$ is the all-ones vector (see \cite[Lemma~3.1]{bhakta4} for details). Then from \eqref{pgst_gamma=1} and the spectral decomposition of $P$, it follows that $\lim_{k \to \infty}T_{\tau_k}(1)E_1\eu=\gamma E_1\ev$.  Since $T_m(1)=1$ for each $m$, we have $E_1\eu=\gamma E_1\ev$.  Noting that $E_1D^\frac{1}{2}\jb=D^\frac{1}{2}\jb$, we have
		\begin{equation}\label{deguv}
			\sqrt{\deg u}=\ip{D^\frac{1}{2}\jb}{\eu}=\ip{D^\frac{1}{2}\jb}{E_1\eu}=\ip{D^\frac{1}{2}\jb}{\gamma E_1\ev}=\ip{D^\frac{1}{2}\jb}{\gamma\ev}=\gamma\sqrt{\deg v}.
		\end{equation}
		This yields $\gamma=1$. Thus (i) implies (ii).

		Now suppose statement (ii) holds.  Then for any $\epsilon>0$, there exists a positive integer $\tau$ such that $\|T_\tau(P)\eu-\ev\|<\epsilon.$ Therefore $|\evt T_\tau(P)\eu-1|=\left|\evt \left(T_\tau(P)\eu-\ev\right)\right|\leq \|T_\tau(P)\eu-\ev\| <\epsilon$.
		
		Finally, assume that statement (iii) holds. Noting that $\left\|\evt  T_\tau(p)\eu|-1\right|\leq|\evt T_\tau(p)\eu-1|,$ we find from (iii) that  for any $\epsilon>0$, there exists a positive integer $\tau$ such that 
		$|\evt T_\tau(p)\eu|>1-\epsilon$. This implies that $\left|\ip{T_\tau(P)\eu}{\ev}\right|>1-\epsilon.$ By Lemma~\ref{ch11}, we have  $\left|\ip{U^\tau N^*\eu}{N^*\ev}\right|>1-\epsilon.$ Then by Lemma~\ref{zhan_pgst}, PGST occurs from $u$ to $v$.
	\end{proof}
	Note that $P$ is a symmetric matrix. Therefore, $T_m(P)$ is also a symmetric matrix for each $m$. Thus, by Lemma \ref{dpgst}, if PGST occurs from vertex $u$ to vertex $v$, then  PGST occurs from $v$ to $u$ as well. Thus, instead of saying PGST occurs from $u$ to $v$, one can say that PGST occurs between $u$ and $v$. The next result follows from \eqref{deguv}.
	\begin{lemma}
		Let $u$ and $v$ be two vertices of a graph $G$. If $G$ exhibits PGST between $u$ and $v$, then $\deg u=\deg v$.
	\end{lemma}
	The Grover walk on a graph is said to be \emph{periodic} if there exists a positive integer $\tau$ such that $U^\tau=I$. We say that a graph is periodic if the Grover walk on it is periodic. The following result establishes that, under periodicity, PST and PGST are equivalent.
	\begin{lemma}\label{pst_pgst}
		Let $G$ be a periodic graph. Then PST occurs between the vertices $u$ and $v$ on $G$ if and only if PGST occurs between $u$ and $v$.
	\end{lemma}
	\begin{proof}
		If PST occurs between the vertices $u$ and $v$, then clearly PGST also occurs between $u$ and $v$.
		
		Conversely, assume that PGST occurs between $u$ and $v$. Since $G$ is periodic, there exists a positive integer $\tau_0$ such that $U^{\tau_0}=I$. Now, define a function $f:\Nl\to \Rl$ by $f(s)=|\evt T_s(P)\eu|$.  Using Lemma~\ref{ch11}, $f(\tau_0+s)=|\evt T_{\tau_0+s}(P)\eu|=|\evt NU^{\tau_0+s}N^*\eu|=|\evt NU^sN^*\eu|=|\evt T_s(P)\eu|=f(s)$.
		Thus $f$ is a periodic function, and hence $f(\Nl)$ is a finite subset of $\Rl$. Since PGST occurs between $u$ and $v$, for each $\epsilon>0$, there exists a positive integer $\tau$ such that 
		$\left| |\evt T_{\tau}(P)\eu|-1\right|<\epsilon.$
		By Lemma \ref{sh12}, for any positive integer $m$, we have
		$|\evt T_{m}(P)\eu|\leq 1.$
		Therefore $\sup_{m\in \Nl} |\evt T_{m}(P)\eu|=1$.	Since $f(\Nl)$ is finite, there exists a positive integer  $m_0$ such that		$|\evt T_{m_0}(P)\eu|=1.$ This implies  that $\left|\ip{U^{m_0}N^*\eu}{N^*\ev}\right|=1$. Therefore by Lemma~\ref{zhan_pst}, PST occurs between $u$ and $v$.
	\end{proof}
	It is known that the cycle $C_n$ is periodic for each positive integer $n$ with $n\geq3$ (see \cite{mixedpaths}). In~\cite{kubota2}, it is proved that the cycle $C_n$ exhibits PST if and only if $n$ is even. Combining these results with Lemma~\ref{pst_pgst}, we obtain the following.
	\begin{lemma}\label{pgst_cycle}
		The cycle on $n$ vertices exhibits PGST if and only if $n$ is even.
	\end{lemma}
	Let $\aut(G)$ be the set of all automorphisms of $G$. An $\eta\in\aut(G)$ corresponds to a permutation matrix  $M_\eta\in\Cl^{V(G)\times V(G)}$ 
	given by $(M_\eta)_{uv}=\delta_{u,\eta(v)}$. For $u\in V(G)$, let $\aut(G)_u=\{\eta\in\aut(G):\eta(u)=u\}$.
	\begin{lemma}[{\cite[Lemma~3.3]{bhakta3}}]\label{per1}
		Let $G$ be a graph with discriminant matrix  $P$ and $\eta\in\aut(G)$. Then $M_\eta P=PM_\eta$. 
	\end{lemma}
	This commuting relation between $M_\eta$ and $P$ has significant implications for PGST. In particular, it implies that automorphisms preserve the occurrence of PGST between vertices, as shown in the following result.
	\begin{lemma}\label{aut_pgst1}
		Let $\eta$ be an automorphism of a graph $G$. Then PGST occurs between the vertices $u$ and $v$ on $G$ if and only if PGST occurs between $\eta(u)$ and $\eta(v)$.
	\end{lemma}
	\begin{proof}
		By Lemma~\ref{dpgst}, if PGST occur between $u$ and $v$, then there exists a sequence  of positive integers $\{\tau_k\}$  such that  $\lim_{k\to \infty} T_{\tau_k}(P)\eu=\ev$. 	This implies that $\lim_{k\to \infty} M_\eta T_{\tau_k}(P)\eu= M_\eta\ev$.	By Lemma \ref{per1}, $M_\eta T_m(P)=T_m(P)M_\eta$ for any positive integer $m$. Thus $\lim_{k\to \infty} T_{\tau_k}(P)M_\eta\eu= M_\eta\ev$.	Since $M_\eta\eu=\mathbf{e}_{\eta(u)}$ and $M_\eta\ev=\mathbf{e}_{\eta(v)}$, we find that     $\lim_{k\to \infty} T_{\tau_k}(P)\mathbf{e}_{\eta(u)}=\mathbf{e}_{\eta(v)}.$ Hence  PGST occurs between $\eta(u)$ and $\eta(v)$. Converse follows from the invertibility of automorphisms of $G$.
	\end{proof}
	The invariance of PGST under automorphisms has further consequences. In particular, if PGST occurs between a pair of vertices $u$ and $v$, then an automorphism fixing  $u$ must also fix $v$.
	\begin{lemma}\label{aut_pgst2}
		Let $u$ and $v$ be two vertices of a graph $G$. If PGST occurs between $u$ and $v$ on $G$, then $\aut(G)_u=\aut(G)_v$.
	\end{lemma}
	\begin{proof}
		Assume that   PGST occurs between $u$ and $v$. Then there exists a sequence  of positive integers $\{\tau_k\}$ such that 
		\begin{equation}\label{auto_pgst}
			\lim_{k\to \infty} T_{\tau_k}(P)\eu=\ev.
		\end{equation}
		Let $\eta\in \aut(G)_u$. Then the corresponding permutation matrix $M_\eta$ satisfies $M_\eta\eu=\eu$.  Thus  from \eqref{auto_pgst}, we have $\lim_{k\to \infty} T_{\tau_k}(P)M_\eta\eu=M_\eta\ev$.
		Since $M_\eta\eu=\eu$, it follows that $\lim_{k\to \infty} T_{\tau_k}(P)\eu=M_\eta\ev$.	Now the sequence $\{T_{\tau_k}(P)\eu\}$ cannot have two different limits. Therefore $M_\eta\ev=\ev$, and hence $\eta(v)=v$. Thus $\eta\in\aut(G)_v$. Similarly, if  $\eta\in\aut(G)_v$ then $\eta\in\aut(G)_u$. Thus we have the desired result.
	\end{proof}
	We now explore a spectral condition that must be satisfied by any pair of vertices exhibiting PGST. The next lemma also appears in \cite{pgst1}. However, our proof technique is completely different. 
	\begin{lemma}\label{pgst_strongly}
		Let $P$ be the discriminant matrix of a graph $G$ with spectral decomposition $P=\sum_{\mu\in\sigma(P)} \mu E_\mu$.  If PGST occurs between the vertices $u$ and $v$,  then $E_\mu\eu=\pm E_\mu\ev$ for each $\mu\in\sigma(P)$.
	\end{lemma}
	\begin{proof}
		Let $m$ be a positive integer and $\mu\in\sigma(P)$.	Note that $|\mu|\leq 1$, and therefore $|T_m(\mu)|\leq1$. We have
		\begin{align}
			\left|\ip{T_m(P)\eu}{\ev}\right|=\left|\sum_{\mu\in\sigma(P)} T_m(\mu)\ip{E_\mu\eu}{E_\mu\ev}\right|&\leq\sum_{\mu\in\sigma(P)} \left|T_m(\mu)\ip{E_\mu\eu}{E_\mu\ev}\right|\nonumber\\
			&\leq\sum_{\mu\in\sigma(P)} \left|\ip{E_\mu\eu}{E_\mu\ev}\right|\nonumber\\
			&\leq \sum_{\mu\in\sigma(P)} \|E_\mu\eu\|\|E_\mu\ev\|\label{cs1}\\
			&\leq \sqrt{\sum_{\mu\in\sigma(P)} \|E_\mu\eu\|^2} \sqrt{\sum_{\mu\in\sigma(P)} \|E_\mu\ev\|^2}\label{cs2}\\
			&=1\nonumber.
		\end{align}
		For each $\mu\in \sigma(P)$, equality holds in \eqref{cs1} if and only if $E_\mu\eu=\alpha E_\mu\ev$ for  some $\alpha\in\Rl$, and equality holds in \eqref{cs2} if and only if $\|E_\mu\eu\|=\|E_\mu\ev\|$. 
		
		Now assume that PGST occurs between the vertices $u$ and $v$. Suppose, if possible,  $E_\mu\eu\neq \pm E_\mu\ev$ for some $\mu\in\sigma(P)$. Then either $E_\mu\eu\neq\alpha E_\mu\ev$ for all $\alpha\in\Rl$ or $\|E_\mu\eu\|\neq\|E_\mu\ev\|$. Therefore, at least one of \eqref{cs1} and \eqref{cs2} must be a strict inequality. Hence $\sum_{\mu\in\sigma(P)} \left|\ip{E_\mu\eu}{E_\mu\ev}\right|<1$. This implies that there exist $\epsilon>0$, we have $\sum_{\mu\in\sigma(P)} \left|\ip{E_\mu\eu}{E_\mu\ev}\right|<1-\epsilon$. Therefore $\left|\ip{T_m(P)\eu}{\ev}\right|<1-\epsilon$ for each $m$. This contradicts the occurrence of PGST. Thus we have the result.
	\end{proof}
	Let the spectral decomposition of $P$ be $P=\sum_{\mu\in\sigma(P)}\mu E_\mu$. The \emph{eigenvalue support} of the vertex $ u $ is defined as $\Theta_u := \Theta_u(P) = \{ \mu\in\sigma(P): E_\mu \mathbf{e}_u \neq 0 \}.$
	The \emph{symmetric} and \emph{antisymmetric} eigenvalue supports between the vertices $u$ and $v$ are defined, respectively, by 
	\[
	\Theta_{uv}^+ := \Theta_{uv}^+(P) = \{ \mu \in \Theta_u : E_\mu \mathbf{e}_u = E_\mu \mathbf{e}_v \}
	\]
	and
	\[
	\Theta_{uv}^- := \Theta_{uv}^-(P) = \{ \mu \in \Theta_u : E_\mu \mathbf{e}_u = -E_\mu \mathbf{e}_v \}.
	\]
	Note that  $E_\mu\eu=\pm E_\mu\ev$ for each $\mu\in\sigma(P)$ if and only if $\Theta_u=\Theta_v=\Theta_{uv}^+\cup \Theta_{uv}^-$.
	\begin{theorem}\label{ch_pgst2}
		Let $u$ and $v$ be two vertices of a graph $G$. Then PGST occurs between $u$ and $v$ if and only if the following two conditions hold.
		\begin{enumerate}[label=(\roman*)]
			\item $E_\mu\eu=\pm E_\mu\ev$ for each $\mu\in\sigma(P)$.
			\item For any $\epsilon>0$ there exists a positive integer $\tau$ such that $|T_\tau(\mu)-1|<\epsilon$ if $\mu\in\Theta_{uv}^+$ and $|T_\tau(\mu)+1|<\epsilon$ if $\mu\in\Theta_{uv}^-$.
		\end{enumerate} 
	\end{theorem}
	\begin{proof}
		Assume that condition~(i) holds. For each  $\mu \in \Theta_u$, let $\alpha_\mu\in\{-1, 1\}$ be such that $E_\mu \eu = \alpha_\mu E_\mu \ev$. Then for any positive integer $m$, we compute
		\begin{align*}
			T_m(P)\eu - \ev 
			&= \sum_{\mu \in \Theta_u} T_m(\mu) E_\mu \eu - \sum_{\mu \in \Theta_u} E_\mu \ev \\
			&= \sum_{\mu \in \Theta_u} \left(T_m(\mu) E_\mu \eu - E_\mu \ev\right) \\
			&= \sum_{\mu \in \Theta_u} \left(T_m(\mu)\alpha_\mu - 1\right) E_\mu \ev \\
			&= \sum_{\mu \in \Theta_u} \alpha_\mu\left(T_m(\mu) - \alpha_\mu\right) E_\mu \ev.
		\end{align*}
		Taking the norm, we find 
		\begin{equation}\label{tmmm}
			\|T_m(P)\eu - \ev\|^2 = \left\| \sum_{\mu \in \Theta_u} \alpha_\mu\left(T_m(\mu) - \alpha_\mu\right) E_\mu \ev \right\|^2 = \sum_{\mu \in \Theta_u} \left|T_m(\mu) - \alpha_\mu\right|^2 \|E_\mu \ev\|^2.
		\end{equation}
		
		Now suppose that PGST occurs between $u$ and $v$. Then by Lemma~\ref{pgst_strongly}, condition (i) holds, and hence \eqref{tmmm} holds as well. Since PGST occurs between $u$ and $v$, for any $\epsilon>0$ there exists a positive integer $\tau$ such that $\|T_\tau(P)\eu-\ev\|<\beta\epsilon$, where $\beta=\min\{\|E_\mu\ev\|: \mu\in\Theta_u \}$. Then form \eqref{tmmm}, we find that  $\left|T_\tau(\mu) - \alpha_\mu\right|<\epsilon$. Thus condition (ii) holds.

		Conversely, assume that conditions (i) and (ii) hold. From condition (ii) and \eqref{tmmm}, we find that for any $\epsilon>0$ there exists a positive integer $\tau$ such that $\|T_\tau(P)\eu-\ev\|<\epsilon$. Hence PGST occurs between $u$ and $v$.
	\end{proof}
	The following variant of the Kronecker approximation theorem is  useful for giving a simple necessary and sufficient condition for the occurrence of PGST on graphs.
	\begin{theorem}[{\cite[Theorem~B]{kronecker}}]\label{kronecker}
		Let $\alpha_1, \ldots, \alpha_n$ and $\beta_1, \ldots, \beta_n$ be any real numbers. Then the following statements are equivalent.
		\begin{enumerate}
			\item[(i)] For each $\epsilon>0$, there exists $\{q, p_1, \ldots, p_n\} \subseteq \mathbb{Z}$ such that
			$|q\alpha_r - \beta_r - p_r| < \epsilon$ for $1\leq r\leq n$.
			
			\item[(ii)] For any set $\{\ell_1, \ldots, \ell_n\}$ of integers, if $\ell_1 \alpha_1 + \cdots + \ell_n \alpha_n \in \mathbb{Z}$, then $
			\ell_1 \beta_1 + \cdots + \ell_n \beta_n \in \mathbb{Z}$.
		\end{enumerate}
	\end{theorem}
	We now provide a number-theoretic characterization of the occurrence of PGST between two vertices of a graph.
	\begin{theorem}\label{pgst_graphs}
		Let $u$ and $v$ be two vertices of a graph $G$. Then PGST occurs between $u$ and $v$ if and only if the following two conditions hold.
		\begin{enumerate}[label=(\roman*)]
			\item $E_\mu\eu=\pm E_\mu\ev$ for each $\mu\in\sigma(P)$.
			\item For any set $\{l_\mu:\mu\in \Theta_u\}$ of integers, if
			$$\sum_{\mu\in \Theta_u}l_\mu\arccos\mu\equiv 0\pmod{2\pi}$$
			then
			$$\sum_{\mu\in\Theta_{uv}^-}l_\mu\equiv 0\pmod{2}.$$
		\end{enumerate} 
	\end{theorem}
	\begin{proof}
		It suffices to show that under condition (i),  PGST occurs between $u$ and $v$ if and only if condition (ii) holds. Assume that condition (i) holds. For each  $\mu\in \Theta_u$, let $\alpha_\mu\in\{0,1\}$ such that $E_\mu\eu=(-1)^{\alpha_\mu}E_\mu\ev$.	By Theorem~\ref{ch_pgst2}, PGST occurs between $u$ and $v$ if and only if for any $\epsilon>0$, there exists a positive integer $\tau$ such that  $|T_\tau(\mu)-\cos(\alpha_\mu\pi)|<\epsilon~\text{for each $\mu\in \Theta_u$}$.
		Note that $T_\tau(\mu)=\cos(\tau\arccos\mu)$. Thus PGST occurs between $u$ and $v$ if and only if  for any $\epsilon>0$, there exists $\{\tau\}\cup\{k_\mu:\mu\in\Theta_u\}\subseteq\Zl$ such that 
		\begin{equation}\label{cosss}
			\left|\tau\frac{\arccos \mu}{2\pi}-\frac{\alpha_\mu}{2}-k_\mu\right|<\epsilon\quad\text{for each }\mu\in\Theta_u.
		\end{equation}
		By Theorem~\ref{kronecker}, the inequalities \eqref{cosss} hold if and only if for any set of integers $\{l_\mu:\mu\in\Theta_u\}$ such that 
		$$\sum_{\mu\in \Theta_u}l_\mu\arccos\mu\equiv 0\pmod{2\pi},$$
		we have
		\[\sum_{\mu\in\Theta_u}l_\mu\alpha_\mu=\sum_{\mu\in\Theta_{uv}^-}l_\mu\equiv 0\pmod 2.\]
		Hence the proof is complete.
	\end{proof}
	We now apply the preceding result to characterize complete graphs that exhibit PGST.
	\begin{lemma}
		A complete graph $K_n$ exhibits PGST if and only if $n=2$.
	\end{lemma}
	\begin{proof}
		The spectral decomposition of the discriminant $P$ of  $K_n$ is 
		$$P=1\cdot E_1+ \left(\frac{1}{1-n}\right)E_\frac{1}{1-n},$$
		where $E_1=\frac{1}{n}J_n$ and $E_\frac{1}{1-n}=I_n-\frac{1}{n}J_n$. If $n\geq 3$, then $E_\frac{1}{1-n}\eu\neq \pm E_\frac{1}{1-n}\ev$ for $u\neq v$. Thus, if $K_n$ exhibits PGST, then $n=2$.
		
		Note that $\sigma\left(P(K_2)\right)=\{-1,1\}$ and $\Theta_{uv}^-=\{-1\}$. Let  $\ell_1$ and $\ell_{-1}$ be two integers. Then one can easily  check that if  $\ell_1 \arccos (1)$ + $\ell_{-1} \arccos(-1)\equiv 0\pmod{2\pi}$ then  $\ell_{-1}\equiv 0\pmod2$. Thus by Theorem~\ref{pgst_graphs}, $K_2$ exhibits PGST.
	\end{proof}
	\subsection{Linear independence of reciprocal-cosine angles}\label{valuation}
	From condition (ii) of  Theorem~\ref{pgst_graphs}, it is evident that PGST is closely related to the linear dependence of certain angles over the rationals. Notably, such linear relations are  well-studied for geodetic angles. A real number $\theta$ is said to be a \emph{pure geodetic angle} if any one of its six squared trigonometric functions is rational or infinite. In this case, all six squared trigonometric functions are rational or infinite.
	
	\begin{lemma}[{\cite[Theorem~3]{conway_angles}}]\label{geo2}
		If the value of a rational linear combination of pure geodetic angles is a rational multiple of $\pi$, then so is the value of its restriction to those angles whose tangents are rational multiples of any given square root.
	\end{lemma}
	
	A positive integer $d$ is called \emph{squarefree} if $p^2\nmid d$ for each prime $p$. For a positive integer $n$, its \emph{squarefree part}, denoted by $\operatorname{sf}(n)$, is the unique squarefree integer $d$ such that $n=db^2$ for some positive integer $b$. A prime $q$ dividing a term $u_j$ of an integer sequence is called a \emph{primitive prime divisor} of $u_j$ if $q\nmid u_1u_2\cdots u_{j-1}$. We first recall the following primitive-divisor theorem.
	
	\begin{lemma}[{\cite[Theorem~1]{yabuta}}]\label{crb1808}
		Let $\ell,m$ be relatively prime nonzero integers with $\ell>0$, and let $\alpha,\beta$ be distinct real roots of $x^2-\ell x+m$. Set $u_j=(\alpha^j-\beta^j)/(\alpha-\beta)$ for $j\geq1$. If $j\notin\{1,2,6\}$, then $u_j$ has a primitive prime divisor except when $j=12$, $\ell=1$, and $m=-1$.
	\end{lemma}
	
	We only use this theorem with $\ell=2x\geq4$ and $m=1$.
	
	\begin{lemma}\label{pell_new_prime}
		Let $d>1$ be squarefree, and let $x_1+y_1\sqrt d$ be the least solution greater than $1$ of $x^2-dy^2=1$ in positive integers. Set $x_r+y_r\sqrt d=(x_1+y_1\sqrt d)^r$ for $r\geq1$. If $r\geq2$, then there is an odd prime $q$ such that $q\mid x_r$ and $q\nmid x_1x_2\cdots x_{r-1}$.
	\end{lemma}
	
	\begin{proof}
		Let $\alpha=x_1+y_1\sqrt d$ and $\beta=\alpha^{-1}=x_1-y_1\sqrt d$, and define $u_j=(\alpha^j-\beta^j)/(\alpha-\beta)$. By conjugation, $\beta^r=x_r-y_r\sqrt d$, so $\alpha^r+\beta^r=2x_r$. Hence
		\[
		u_{2r}
		=\frac{\alpha^{2r}-\beta^{2r}}{\alpha-\beta}
		=(\alpha^r+\beta^r)u_r
		=2x_ru_r.
		\]
		Suppose first that $r\ne3$. The numbers $\alpha$ and $\beta$ are distinct real roots of $t^2-2x_1t+1$. Applying Lemma~\ref{crb1808} with $\ell=2x_1$, $m=1$, and $j=2r$ gives a primitive prime divisor $q$ of $u_{2r}$, since $2r\notin\{1,2,6\}$ and the exception at $j=12$ requires $m=-1$. Since $2\mid u_2$, the prime $q$ is odd. Also, $q\nmid u_r$, so $q\mid x_r$. If $q\mid x_s$ for some $s<r$, then $q\mid 2x_su_s=u_{2s}$, contrary to primitivity.
		
		It remains to consider $r=3$. Expanding the powers and using $dy_1^2=x_1^2-1$, we obtain $x_2=2x_1^2-1$ and $x_3=x_1(4x_1^2-3)$. We have $4x_1^2-3\equiv1\pmod3$ if $3\nmid x_1$, and $4x_1^2-3\equiv-3\pmod9$ otherwise. Thus $4x_1^2-3$ is divisible by $3$ at most once. Since it is odd and exceeds $3$, it has a prime divisor $q\ne3$. Now $\gcd(x_1,4x_1^2-3)\mid3$, and $4x_1^2-3=2x_2-1$ is coprime to $x_2$. Hence $q$ is an odd prime dividing $x_3$ with $q\nmid x_1x_2$.
	\end{proof}
	
	Let $K$ be a number field, and let $\mathcal O_K$ be its ring of integers, consisting of the elements of $K$ that satisfy a monic polynomial with integer coefficients. A nonzero prime ideal $\mathfrak q$ of $\mathcal O_K$ is said to lie above a prime integer $q$ if $q\in\mathfrak q$. In this case, $\mathfrak q\cap\mathbb Z=q\mathbb Z$.
	
	For $0\ne z\in K$, the fractional ideal $z\mathcal O_K=\{z\alpha:\alpha\in\mathcal O_K\}$ has a unique factorization
	\[
	z\mathcal O_K=\prod_{\mathfrak p}\mathfrak p^{n_{\mathfrak p}},
	\]
	where $\mathfrak p$ runs over the nonzero prime ideals of $\mathcal O_K$, and the exponents are integers, only finitely many of which are nonzero. We define $v_{\mathfrak q}(z)=n_{\mathfrak q}$ and call it the valuation of $z$ at $\mathfrak q$.  For nonzero $u,v\in K$, we have $v_{\mathfrak q}(uv)=v_{\mathfrak q}(u)+v_{\mathfrak q}(v)$. If $u+v\ne0$, then $v_{\mathfrak q}(u+v)\geq\min\{v_{\mathfrak q}(u),v_{\mathfrak q}(v)\}$, with equality whenever $v_{\mathfrak q}(u)\ne v_{\mathfrak q}(v)$.
	
	In particular, if $K=\Ql$ then we write $v_q$ for the valuation at $q\Zl$. For a nonzero integer $m$, $v_q(m)$ is the largest integer $k\geq0$ such that $q^k\mid m$. For a nonzero rational number $a/b$, we have $v_q(a/b)=v_q(a)-v_q(b)$.  We call $v_q$ the $q$-adic valuation. 
	
	If $\mathfrak q$ lies above a rational prime $q$, the exponent of $\mathfrak q$ in the prime ideal factorization of $q\mathcal O_K$ is called the \emph{ramification index}, denoted by $e(\mathfrak q/q)$. For every $a\in\mathbb Q^\times$, we have $v_{\mathfrak q}(a)=e(\mathfrak q/q)\,v_q(a)$.	Since $e(\mathfrak q/q)$ is a positive integer, these two valuations have the same sign on nonzero rational numbers. We refer to \cite{ant_stewart} for more details.

	\begin{theorem}\label{li_arccos_i}
		Let $a_1,\ldots,a_s$ be distinct integers with $a_i\geq3$. Then $\{\pi\}\cup\{\arccos(1/a_i):1\leq i\leq s\}$ is linearly independent over $\mathbb Q$.
	\end{theorem}
	
	\begin{proof}
		Write $\theta_{a_i}=\arccos(1/a_i)$. It suffices to show that if $\sum_{i=1}^s c_i\theta_{a_i}\in\pi\mathbb Q$ with $c_i\in\mathbb Z$, then $c_i=0$ for $1\leq i\leq s$.
		
		For $1\leq i\leq s$, write $a_i^2-1=d_i b_i^2$, where $d_i=\operatorname{sf}(a_i^2-1)$ and $b_i$ is a positive integer. Since $(a_i-1)^2<a_i^2-1<a_i^2$, the case $d_i=1$ would give $a_i-1<b_i<a_i$, which is impossible for an integer $b_i$. Thus $d_i>1$. Since $\cos^2\theta_{a_i}=1/a_i^2\in\mathbb Q$, each $\theta_{a_i}$ is pure geodetic, and $\tan\theta_{a_i}=b_i\sqrt{d_i}$.
		
		Fix a value $d$ among the $d_i$, and let $1\leq i_1<\cdots<i_k\leq s$ be precisely the indices for which $d_i=d$. By Lemma~\ref{geo2}, it follows that
		\begin{equation}\label{crb_rest}
			\sum_{j=1}^k c_{i_j}\theta_{a_{i_j}}\in\pi\mathbb Q.
		\end{equation}
		By the classical theory of Pell equations, every positive integer solution of $x^2-dy^2=1$ is obtained from a positive integer power of its fundamental solution. Hence $(a_{i_j},b_{i_j})=(x_{r_j},y_{r_j})$ for some $r_j\geq1$, where $x_\ell,y_\ell$ are defined as in Lemma~\ref{pell_new_prime}. The indices $r_j$ are distinct because the $a_i$ are distinct.
		
		In $K=\mathbb Q(\sqrt{-d})$, taking $\sqrt{-d}=i\sqrt d$, put $z_\ell=(1+y_\ell\sqrt{-d})/x_\ell$. Since $\cos\theta_{a_{i_j}}=1/x_{r_j}$ and $\sin\theta_{a_{i_j}}=y_{r_j}\sqrt d/x_{r_j}$, we have $z_{r_j}=e^{i\theta_{a_{i_j}}}$. Also, $x_\ell^2=1+dy_\ell^2$ gives $z_\ell^{-1}=(1-y_\ell\sqrt{-d})/x_\ell$, so $z_\ell+z_\ell^{-1}=2/x_\ell$. Exponentiating the relation in \eqref{crb_rest}, we obtain
		\begin{equation}\label{crb2338}
			\prod_{j=1}^k z_{r_j}^{c_{i_j}}=\zeta,
		\end{equation}
		where $\zeta=e^{i\pi t}\in K$ for some $t\in\mathbb Q$. Choose a positive integer $\kappa$ such that $\kappa t\in\mathbb Z$. Then $\zeta^{2\kappa}=1$. 
		
		Suppose some $c_{i_j}\ne0$, and let $r=r_{j_0}$ be the largest $r_j$ with $c_{i_j}\ne0$. If $r\geq2$, Lemma~\ref{pell_new_prime} gives an odd prime $q$ dividing $x_r$ and none of $x_1,\ldots,x_{r-1}$. If $r=1$, then $x_1=a_{i_{j_0}}\geq3$, so we may choose a prime $q$ dividing the denominator of $2/x_1$ in lowest terms.  Fix a prime ideal $\mathfrak q$ of $\mathcal O_K$ above $q$. In either case, $v_{\mathfrak q}(2/x_r)<0$ and $v_{\mathfrak q}(2/x_\ell)\geq0$ for $1\leq\ell<r$.

		Since $v_{\mathfrak q}(z_\ell^{-1})=-v_{\mathfrak q}(z_\ell)$, if $v_{\mathfrak q}(z_\ell)\ne0$, then $v_{\mathfrak q}(z_\ell+z_\ell^{-1})=\min\{v_{\mathfrak q}(z_\ell),-v_{\mathfrak q}(z_\ell)\}=-|v_{\mathfrak q}(z_\ell)|<0$.	If $v_{\mathfrak q}(z_\ell)=0$, the valuation inequality gives $v_{\mathfrak q}(z_\ell+z_\ell^{-1})\geq0$. Since $z_\ell+z_\ell^{-1}=2/x_\ell$, the inequalities established above imply that $v_{\mathfrak q}(z_r)\ne0$ and $v_{\mathfrak q}(z_\ell)=0$ for $1\leq\ell<r$.
		
		Since $\zeta^{2\kappa}=1$, we have $2\kappa v_{\mathfrak q}(\zeta)=v_{\mathfrak q}(1)=0$, and therefore $v_{\mathfrak q}(\zeta)=0$. Terms with $r_j<r$ have zero valuation, whereas terms with $r_j>r$ have zero coefficient by the choice of $r$. Taking valuations in \eqref{crb2338}, we obtain
		\[
		0=v_{\mathfrak q}(\zeta)
		=\sum_{j=1}^k c_{i_j}v_{\mathfrak q}(z_{r_j})
		=c_{i_{j_0}}v_{\mathfrak q}(z_r),
		\]
		contradicting $c_{i_{j_0}}\neq0$ and $v_{\mathfrak q}(z_r)\neq0$. Thus $c_{i_j}=0$ for every $1\leq j\leq k$. Applying this argument to each value of $d$ gives $c_i=0$ for $1\leq i\leq s$.
	\end{proof}
	We recall the following density form of Kronecker's theorem, which allows simultaneous approximation of several phases. Write $\mathbb T=\{z\in\mathbb C:|z|=1\}$.
	\begin{lemma}[{\cite[Corollary~2]{ouaknine_density}}]\label{kronecker_density}
		If $1,\alpha_1,\ldots,\alpha_s$ are real numbers linearly independent over $\mathbb Q$, then
		$\left\{\bigl(e^{2\pi i q\alpha_1},\ldots,e^{2\pi i q\alpha_s}\bigr):q\in\Nl\right\}$ is dense in $\mathbb T^s$.
	\end{lemma}
	Together with Theorem~\ref{li_arccos_i}, the preceding result gives simultaneous approximation of any prescribed phases along any arithmetic progression.
	
	\begin{coro}\label{progression_density}
		Let $a_1,\ldots,a_s\geq3$ be distinct integers. For any positive integer $\ell$, integer $b$, and real numbers $\beta_1,\ldots,\beta_s$, there exists a sequence of positive integers $t_k\equiv b\pmod\ell$ such that
		\[
		\lim_{k \to \infty}\exp\!\left(it_k\arccos\frac1{a_j}\right)=e^{i\beta_j}
		\quad\text{for }1\leq j\leq s.
		\]
	\end{coro}
	
	\begin{proof}
		Write $\theta_{a_j}=\arccos(1/a_j)$ and put $\alpha_j=\ell\theta_{a_j}/(2\pi)$. By Theorem~\ref{li_arccos_i}, the numbers $1,\alpha_1,\ldots,\alpha_s$ are linearly independent over $\mathbb Q$. Therefore by Lemma~\ref{kronecker_density}, the sequence $(e^{i\ell q\theta_{a_1}},\ldots,e^{i\ell q\theta_{a_s}})$  is dense in $\mathbb T^s$. Removing finitely many terms preserves density. Thus, for each $k$, we can choose an integer $q_k>k$ such that $|e^{i\ell q_k\theta_{a_j}}-e^{i(\beta_j-b\theta_{a_j})}|<1/k$ for every $j$. Hence $\lim_{k \to \infty}q_k=\infty$ and $\lim_{k \to \infty}e^{i\ell q_k\theta_{a_j}}=e^{i(\beta_j-b\theta_{a_j})}$. Put $t_k=\ell q_k+b$. Then $t_k\equiv b\pmod\ell$ and $t_k>0$ for all sufficiently large $k$. Moreover,
		\[
		\lim_{k \to \infty}e^{it_k\theta_{a_j}}=\lim_{k \to \infty}e^{ib\theta_{a_j}}e^{i\ell q_k\theta_{a_j}}=e^{i\beta_j}.
		\]
		After removing finitely many initial terms, all $t_k$ are positive, and the required convergence holds for $1\leq j\leq s$.
	\end{proof}

	\subsection{Independent set blow-ups}
	State transfer on independent-set blow-ups in continuous-time quantum walks was studied by Bhattacharjya, Monterde and Pal~\cite{bmp_blowups}. In this section, we study PGST on blow-ups in Grover walks.
	
	Let $G$ be a graph without isolated vertices. For an integer $m\geq2$, the blow-up $G[\overline K_m]$ is obtained by replacing each vertex of $G$ by an independent set of size $m$ and joining every vertex in one set to every vertex in another whenever the corresponding vertices of $G$ are adjacent. These independent sets are called the \emph{fibres}. The graph $G[\overline K_m]$ is the lexicographic product of $G$ and $\overline K_m$.
	
	Each vertex in the fibre corresponding to $u\in V(G)$ has degree $m\deg_G(u)$. Thus, ordering the vertices by fibres, we obtain
	\begin{equation}\label{blowup_discriminant}
		P(G[\overline K_m])=P(G)\otimes\frac1mJ_m,
	\end{equation}
	where $J_m$ is the $m\times m$ all-ones matrix.
	\begin{lemma}\label{blowup_obstruction}
		Let $G$ be a graph without isolated vertices and let $m\geq2$. If $G[\overline K_m]$ exhibits PGST, then $m=2$, and these vertices belong to the same fibre.
	\end{lemma}
	
	\begin{proof}
		Suppose PGST occurs between distinct vertices $u$ and $v$. By Lemma~\ref{aut_pgst2}, each automorphism fixing $u$ must also fix $v$. For any vertex $w\ne v$ in the fibre of $v$, define $\sigma(v)=w$, $\sigma(w)=v$, and $\sigma(x)=x$ for $x\notin\{v,w\}$. Since $v$ and $w$ have the same neighbours, $\sigma$ is an automorphism.
		
		If $u$ and $v$ belong to different fibres, then $\sigma(u)=u$ and $\sigma(v)=w\ne v$, a contradiction. Thus they belong to the same fibre. If $m\geq3$, choose $w$ in this fibre with $w\notin\{u,v\}$. Again, $\sigma$ fixes $u$, but does not $v$, a contradiction. Hence $m=2$.
	\end{proof}
	
	Let $G$ be a graph without isolated vertices. Let $V(\overline K_2)=\{0,1\}$. Then $V(G[\overline K_2])=V(G)\times\{0,1\}$, and the fibre corresponding to $u\in V(G)$ is $\{(u,0),(u,1)\}$. Write $P(G)=\sum_{\mu\in\sigma(P(G))}\mu E_\mu$, and let $\Theta_u(G)=\{\mu\in\sigma(P(G)):E_\mu\eu\ne0\}$.
	
	\begin{lemma}\label{blowup_exact}
		Let $u\in V(G)$. Then PGST occurs between $(u,0)$ and $(u,1)$ in $G[\overline K_2]$ if and only if $0\notin\Theta_u(G)$ and there exists a sequence of positive integers $t_k\equiv2\pmod4$ such that $T_{t_k}(\mu)$ converges to $1$ for each $\mu\in\Theta_u(G)$.
	\end{lemma}
	
	\begin{proof}
		Let $Q_1=\frac12J_2$ and $Q_2=I_2-Q_1$. The matrix $\frac 12J_2$ has eigenvalues $1$ and $0$, with spectral projections $Q_1$ and $Q_2$, respectively. By \eqref{blowup_discriminant}, the spectral decomposition of $P(G[\overline K_2])$ is
		\[
		P(G[\overline K_2])
		=\sum_{\mu\in\sigma(P(G))\setminus\{0\}}\mu F_\mu+0\cdot F_0,
		\]
		where $F_\mu=E_\mu\otimes Q_1$ for $\mu\ne0$ and $F_0=E_0\otimes Q_1+I\otimes Q_2$, with the convention $E_0=0$ if $0\notin\sigma(P(G))$.

		For $j\in\{0,1\}$, we have $\e_{(u,j)}=\eu\otimes\e_j$. Since $Q_1\e_0=Q_1\e_1$, we have $F_\mu\e_{(u,0)}=F_\mu\e_{(u,1)}$ for each nonzero eigenvalue $\mu$ of $P(G)$. For the eigenvalue $0$ of $G[\overline K_2]$, we have
		\[
		\begin{aligned}
			F_0\e_{(u,0)}-F_0\e_{(u,1)}&=F_0\bigl(\mathbf e_u\otimes(\mathbf e_0-\mathbf e_1)\bigr)=E_0\mathbf e_u\otimes0
			+\mathbf e_u\otimes(\mathbf e_0-\mathbf e_1)=\mathbf e_u\otimes(\mathbf e_0-\mathbf e_1)\ne0,\\
			F_0\e_{(u,0)}+F_0\e_{(u,1)}
			&=F_0\bigl(\mathbf e_u\otimes(\mathbf e_0+\mathbf e_1)\bigr)=E_0\mathbf e_u\otimes(\mathbf e_0+\mathbf e_1)
			+\mathbf e_u\otimes0=E_0\mathbf e_u\otimes(\mathbf e_0+\mathbf e_1).
		\end{aligned}
		\]
		The first identity gives $F_0\e_{(u,0)}\ne F_0\e_{(u,1)}$. Since $\e_0+\e_1\ne0$, the second shows that $F_0\e_{(u,0)}=-F_0\e_{(u,1)}$ if and only if $E_0\eu=0$. Hence condition~(i) of Theorem~\ref{ch_pgst2} holds exactly when $0\notin\Theta_u(G)$.
		
		Suppose $0\notin\Theta_u(G)$. Then the symmetric and antisymmetric eigenvalue supports of $(u,0)$ and $(u,1)$ are $\Theta_u(G)$ and $\{0\}$, respectively. By Theorem~\ref{ch_pgst2}, PGST occurs exactly when there is a sequence of positive integers $t_k$ such that $\lim_{k \to \infty}T_{t_k}(\mu)=1$ for each $\mu\in\Theta_u(G)$ and $\lim_{k \to \infty}T_{t_k}(0)=-1$. Since $\lim_{k \to \infty}T_{t_k}(0)=-1$ if and only if  $t_k\equiv2\pmod4$, we have the result.
	\end{proof}
	
	\begin{theorem}\label{reciprocal_blowup}
		Let $G$ be a graph with $\sigma(P(G))\subseteq\{1/a:a\in\Zl\setminus\{0\}\}$. Then PGST occurs between $(u,0)$ and $(u,1)$ in $G[\overline K_2]$ for each $u\in V(G)$. 
	\end{theorem}
	
	\begin{proof}
		The hypothesis gives $0\notin\sigma(P(G))$, so $0\notin\Theta_u(G)$ for each $u\in V(G)$. If $t\equiv6\pmod{12}$, then $t\equiv2\pmod4$ and $T_t(\pm1)=T_t(\pm1/2)=1$, since $T_t(\cos\theta)=\cos(t\theta)$.
		
		Let $a_1,\ldots,a_s\geq3$ be the distinct positive integers for which $1/a_j$ or $-1/a_j$ belongs to $\sigma(P(G))$. By Corollary~\ref{progression_density}, with $\ell=12$, $b=6$, and $\beta_j=0$, there is a sequence of positive integers $t_k\equiv6\pmod{12}$ such that $e^{it_k\arccos(1/a_j)}\to1$ for each $j$. Taking real parts gives $T_{t_k}(1/a_j)\to1$. 
		
		If $\sigma(P(G))\subseteq\{\pm1,\pm\frac12\}$, take $t_k=12k+6$ for $k\geq1$. Then $t_k\equiv2\pmod4$ and $T_{t_k}(\mu)=1$ for each $\mu\in\sigma(P(G))$.
		
		Since each $t_k$ is even, $T_{t_k}(-\mu)=T_{t_k}(\mu)$. Thus $T_{t_k}(\mu)\to1$ for each $\mu\in\sigma(P(G))$, and therefore Lemma~\ref{blowup_exact} gives PGST between $(u,0)$ and $(u,1)$ for each $u$.
	\end{proof}
	
	\begin{coro}\label{complete_tensor_blowup}
		Let $q_1,\ldots,q_s$ be integers with $q_i\geq2$, and let $G=K_{q_1}\otimes\cdots\otimes K_{q_s}$, where $\otimes$ denotes the tensor product of graphs. Then $G[\overline K_2]$ exhibits PGST.
	\end{coro}	
	\begin{proof}
		For each $i$, the discriminant eigenvalues of $K_{q_i}$ are $1$ and $-1/(q_i-1)$. Since $P(G)=P(K_{q_1})\otimes\cdots\otimes P(K_{q_s})$, the eigenvalues of $P(G)$ are
		\[
		\frac{(-1)^{|X|}}{\prod_{i\in X}(q_i-1)}
		\quad\text{for}\quad X\subseteq\{1,\ldots,s\},
		\]
		where the empty product is $1$, therefore the result follows from Theorem~\ref{reciprocal_blowup}.
	\end{proof}
	For $s=1$, Corollary~\ref{complete_tensor_blowup} recovers the PGST result of Chan and Zhan~\cite[Theorem~7.4]{pgst1} for cocktail party graphs. Thus, the corollary extends their result to blow-ups of tensor products of complete graphs. 
	\section{Pretty good state transfer on Cayley graphs}
	Let $\Gamma$ be a finite group with identity element $1$. Let $S\subseteq\Gamma\setminus\{1\}$ satisfy $S=S^{-1}$ and $\langle S\rangle=\Gamma$. The Cayley graph $G:=\cay(\Gamma,S)$ has vertex set $\Gamma$, with $u$ and $v$ adjacent if $uv^{-1}\in S$.
	
	\begin{lemma}\label{pgst_permutation}
		The Cayley graph $\cay(\Gamma,S)$ exhibits PGST if and only if there exists a sequence of positive integers $\{\tau_k\}$ such that $T_{\tau_k}(P)$ converges to a permutation matrix of order $2$ with no fixed points.
	\end{lemma}
	
	\begin{proof}
		Suppose PGST occurs between distinct vertices $u$ and $v$. By Lemma~\ref{dpgst}, there exists a sequence of positive integers $\{\tau_k\}$ such that $\lim_{k \to \infty}T_{\tau_k}(P)\eu=\ev$.
		
		For each $g\in\Gamma$, the map $\rho_g(x)=xg$ is an automorphism. By Lemma~\ref{per1}, the permutation matrix $M_{\rho_g}$ of $\rho_g$ commutes with $P$, and hence with $T_m(P)$ for each $m$. Therefore 
		\[\lim_{k \to \infty}T_{\tau_k}(P)\mathbf e_{ug}=\lim_{k \to \infty}M_{\rho_g}T_{\tau_k}(P)\eu= M_{\rho_g}\ev=\mathbf e_{vg}.\]
		For each $x\in\Gamma$, taking $g=u^{-1}x$ gives $\lim_{k \to \infty}T_{\tau_k}(P)\mathbf e_x=\mathbf e_{vu^{-1}x}$. Thus each column converges along the same sequence. Since $\Gamma$ is finite, $\lim_{k \to \infty}T_{\tau_k}(P)= B$, where $B\mathbf e_x=\mathbf e_{vu^{-1}x}$ for each $x\in\Gamma$. The map $x\mapsto vu^{-1}x$ is a permutation of $\Gamma$, and it has no fixed points since $u\ne v$. Thus $B$ is a permutation matrix with no fixed points. Also, each $T_{\tau_k}(P)$ is symmetric, so $B^t=B$. Hence $B^2=B^tB=I$. Since $B\ne I$, its order is $2$.
		
		Conversely, suppose $\lim_{k \to \infty}T_{\tau_k}(P)= B$, where $B$ is a permutation matrix of order $2$ with no fixed points. Choose $u\in\Gamma$ and let $B\eu=\ev$. Then $u\ne v$ and $\lim_{k \to \infty}T_{\tau_k}(P)\eu=\ev$. By Lemma~\ref{dpgst}, PGST occurs between $u$ and $v$.
	\end{proof}
	
	\begin{lemma}\label{ord_v-u}
		If $\cay(\Gamma,S)$ exhibits PGST between $u$ and $v$, then $vu^{-1}$ has order $2$ in $\Gamma$.
	\end{lemma}
	
	\begin{proof}
		By the proof of Lemma~\ref{pgst_permutation}, the corresponding limiting permutation matrix $B$ satisfies $B\mathbf e_x=\mathbf e_{vu^{-1}x}$ for each $x\in\Gamma$ and $B^2=I$. In particular, $\mathbf e_1=B^2\mathbf e_1=\mathbf e_{(vu^{-1})^2}$, and hence $(vu^{-1})^2=1$. Since $u\ne v$, we have $vu^{-1}\ne1$. Hence $vu^{-1}$ has order $2$.
	\end{proof}
	\subsection{Abelian Cayley graphs}
	We first recall the classification of PST in Grover walks on abelian Cayley graphs obtained in~\cite{bhakta_abelian}. Two pairs $(\Gamma,S)$ and $(\Gamma',S')$ are called \emph{equivalent} if there exists a group isomorphism $\varphi:\Gamma\longrightarrow\Gamma'$ such that $\varphi(S)=S'$.
	\begin{theorem}[{\cite[Theorem~1.4]{bhakta_abelian}}]\label{complete_classification}
	Let $\Gamma$ be a finite abelian group, and let $S\subseteq\Gamma\setminus\{0\}$ satisfy $S=-S$ and $\langle S\rangle=\Gamma$. Then $\cay(\Gamma,S)$ exhibits perfect state transfer if and only if $(\Gamma,S)$ is equivalent to one of the following pairs.
	\begin{enumerate}[label=\textup{(\roman*)}]
		\item $\bigl(\Zl_2,\{1\}\bigr)$;	
		\item $\bigl(\Zl_2\times\Zl_2,\{(1,0),(0,1)\}\bigr)$;
		\item $\bigl(\Zl_{2n},\{1,-1\}\bigr)$, where $n\geq2$;
		\item $\bigl(\Zl_{2m},\{1,-1,m+1,m-1\}\bigr)$, where
		$m\geq3$ and $4$ does not divide $m$; or
		\item $\bigl(\Zl_m\times\Zl_2,
		\{(1,0),(-1,0),(1,1),(-1,1)\}\bigr)$, where
		$m\geq6$ and $m\equiv2\pmod4$.
	\end{enumerate}
	\end{theorem}
	Since PST implies PGST, every graph in Theorem~\ref{complete_classification} also exhibits PGST. The theorem restricts PST to valencies $1$, $2$, and $4$. In contrast,  abelian Cayley graphs exhibiting PGST exist for every positive valency. Indeed, $K_2$, $C_4$, and $K_3[\overline K_2]$ provide examples of valencies $1$, $2$, and $4$, respectively, while Theorem~\ref{prism_independent} gives the $d$-regular example $K_d\square K_2$ for each $d\geq3$ with $d\ne4$.
	
	Thus the bounded-valency reduction available for PST does not extend to PGST. Moreover, the occurrence of PGST depends on integer relations among the spectral angles, as shown in Theorem~\ref{pgst_graphs}. We therefore give a spectral and number-theoretic characterization for abelian Cayley graphs and obtain explicit classifications for several families.

	Let $\Gamma$ be a finite abelian group with the identity element $0$. Let $\hat{\Gamma}$ denote the character group of $\Gamma$. It is well known that the eigenvalues and eigenvectors of $\cay(\Gamma,S)$ can be determined using the characters of $\Gamma$. We now describe these characters explicitly. By the fundamental theorem of finite abelian groups, $\Gamma$ decomposes into a direct sum of cyclic groups: $\Gamma=\mathbb Z_{n_1}\oplus\cdots\oplus\mathbb Z_{n_k}$, where $n_r\geq2$. For each $a=(a_1,\ldots,a_k)\in\Gamma$, define the map $\chi_a:\Gamma\rightarrow\mathbb C$ by
	\begin{equation}\label{character}
		\chi_a(b)=\prod_{r=1}^k\omega_{n_r}^{a_rb_r}\quad\text{for }b=(b_1,\ldots,b_k)\in\Gamma,
	\end{equation}
	where $\omega_{n_r}=\exp(2\pi i/n_r)$. Each such map $\chi_a$ is a character of $\Gamma$. Moreover, the map $\Gamma\to\hat{\Gamma}$ given by $a\mapsto\chi_a$ is a group isomorphism. Hence the character group is given by $\hat{\Gamma}=\{\chi_a\mid a\in\Gamma\}$, and the characters satisfy the symmetry $\chi_a(b)=\chi_b(a)$ for all $a,b\in\Gamma$.
	
	For each $a\in\Gamma$, the adjacency matrix $A$ of $\cay(\Gamma,S)$ has eigenvalue $\ld_a=\sum_{z\in S}\chi_a(z)$ with corresponding eigenvector $\psi_a=\frac{1}{\sqrt{|\Gamma|}}[\chi_a(x)]^t_{x\in\Gamma}$. We refer to \cite{rep} for more details about the eigenvalues and eigenvectors of Cayley graphs.
	
	Let $n=|\Gamma|$ and $s=|S|$. Then for each $a\in\Gamma$, the discriminant matrix $P$ of $\cay(\Gamma,S)$ has eigenvalue
	\begin{equation}\label{evalue_cayley}
		\mu_a=\frac1s\sum_{z\in S}\chi_a(z)
	\end{equation}
	with the corresponding unit eigenvector
	\begin{equation}\label{evector_cayley}
		\psi_a=\frac1{\sqrt n}[\chi_a(x)]^t_{x\in\Gamma}.
	\end{equation}
	Thus the spectral decomposition of $P$ is $P=\sum_{a\in\Gamma}\mu_a\psi_a\psi_a^*$. Moreover, if $f(x)$ is any polynomial, then we have
	\begin{equation}\label{sd_cayley}
		f(P)=\sum_{a\in\Gamma}f(\mu_a)\psi_a\psi_a^*.
	\end{equation}
	
	Let $u$ and $v$ be two distinct vertices of $\cay(\Gamma,S)$, and let $w=v-u$. By Lemma~\ref{ord_v-u}, if $\cay(\Gamma,S)$ exhibits PGST between $u$ and $v$, then the order of $w$ is $2$. Consequently, $|\Gamma|$ is even, and for each character $\chi_a$ of $\Gamma$, we have $\chi_a(w)=\pm1$. Whenever $w$ has order $2$, write $\chi_a(w)=(-1)^{\alpha_a}$ for some $\alpha_a\in\{0,1\}$, and define the subsets $\Gamma_{m,w}=\{a\in\Gamma:\alpha_a=m\}$ for $m\in\{0,1\}$. Note that $\Gamma=\Gamma_{0,w}\cup\Gamma_{1,w}$.
	
	\begin{theorem}\label{pgst_cayley}
		Let $\Gamma$ be a finite abelian group and let $S\subset\Gamma\setminus\{0\}$ with $S=-S$ and $\langle S\rangle=\Gamma$. Then the Cayley graph $\cay(\Gamma,S)$ exhibits PGST between $u$ and $v$ if and only if the following two conditions are satisfied.
		\begin{enumerate}
			\item[(i)] The order of $w:=v-u$ is $2$.
			\item[(ii)] For any set $\{\ell_a:a\in\Gamma\}$ of integers, if
			\[
			\sum_{a\in\Gamma}\ell_a\arccos\mu_a\equiv0\pmod{2\pi},
			\]
			then
			\[
			\sum_{a\in\Gamma_{1,w}}\ell_a\equiv0\pmod2.
			\]
		\end{enumerate}
	\end{theorem}
	
	\begin{proof}
		By Lemma~\ref{ord_v-u}, condition (i) is necessary. By Lemma~\ref{pgst_permutation}, if PGST occurs between $u$ and $v$, then there exists a sequence of positive integers $\{\tau_k\}$ such that $T_{\tau_k}(P)\to B$, where $B$ is a permutation matrix with $\evt B\eu=1$. Since each $T_{\tau_k}(P)$ is symmetric, $B^t=B$.
		
		For each $a\in\Gamma$, we have $T_{\tau_k}(P)\psi_a\to B\psi_a$, or equivalently, $T_{\tau_k}(\mu_a)\psi_a\to B\psi_a$. Since $B^t=B$ and $\evt B\eu=1$, the $u$-th entry of $B\psi_a$ is the $v$-th entry of $\psi_a$. Comparing the $u$-th entries therefore gives $\lim_{k\to\infty}T_{\tau_k}(\mu_a)=\chi_a(v)/\chi_a(u)=\chi_a(w)$. Thus
		\begin{equation}\label{num}
			\lim_{k\to\infty}T_{\tau_k}(\mu_a)=\cos(\alpha_a\pi)
			\quad\text{for each }a\in\Gamma.
		\end{equation}
		Since $T_\tau(\mu_a)=\cos(\tau\arccos\mu_a)$, it follows from \eqref{num} that for any $\epsilon>0$, there exist integers $\tau$ and $k_a$, for $a\in\Gamma$, such that
		\[
		\left|\tau\frac{\arccos\mu_a}{2\pi}-\frac{\alpha_a}{2}-k_a\right|<\epsilon
		\quad\text{for each }a\in\Gamma.
		\]
		Therefore by Theorem~\ref{kronecker}, for any set of integers $\{\ell_a:a\in\Gamma\}$ satisfying $\sum_{a\in\Gamma}\ell_a\arccos\mu_a\equiv0\pmod{2\pi}$, we have
		\[
		\sum_{a\in\Gamma}\ell_a\alpha_a
		=\sum_{a\in\Gamma_{1,w}}\ell_a
		\equiv0\pmod2.
		\]
		Hence condition (ii) holds.
		
		Conversely, let conditions (i) and (ii) hold. By Theorem~\ref{kronecker}, for any $\epsilon>0$, there exist integers $\tau$ and $k_a$, for $a\in\Gamma$, such that
		\[
		\left|\tau\frac{\arccos\mu_a}{2\pi}-\frac{\alpha_a}{2}-k_a\right|<\epsilon
		\quad\text{for each }a\in\Gamma.
		\]
		Since characters separate the elements of $\Gamma$ and $w\ne0$, some $\alpha_a=1$. Thus $\tau\ne0$ whenever $\epsilon<1/2$. Taking $\epsilon=1/(k+2)$ and replacing each resulting $\tau$ by $|\tau|$, using that cosine is even, gives a common sequence of positive integers $\{\tau_k\}$ such that $\lim_{k\to\infty}T_{\tau_k}(\mu_a)=\cos(\alpha_a\pi)$ for each $a\in\Gamma$.
		
		From \eqref{sd_cayley}, we have
		\[
		\evt T_{\tau_k}(P)\eu
		=\frac1n\sum_{a\in\Gamma}T_{\tau_k}(\mu_a)\chi_a(w).
		\]
		Since $\cos(\alpha_a\pi)=\chi_a(w)$ and $\chi_a(w)^2=1$, it follows that $\lim_{k\to\infty}\evt T_{\tau_k}(P)\eu=1$. Hence by Lemma~\ref{dpgst}, PGST occurs between $u$ and $v$.
	\end{proof}
	\begin{coro}\label{pgst_circulant1}
		The circulant graph $\cay(\Zl_n,S)$ exhibits PGST between $u$ and $v$ if and only if the following two conditions are satisfied.
		\begin{enumerate}
			\item[(i)] $n$ is even and $v-u=n/2$.
			\item[(ii)] For any set $\{\ell_r:0\leq r\leq n-1\}$ of integers, if
			\[
			\sum_{r=0}^{n-1}\ell_r\arccos\mu_r\equiv0\pmod{2\pi},
			\]
			then
			\[
			\sum_{\substack{r=0\\r\text{ odd}}}^{n-1}\ell_r\equiv0\pmod2.
			\]
		\end{enumerate}
	\end{coro}
	For the circulant graph $\cay(\Zl_n,S)$, the characters are $\chi_j(x)=\omega_n^{jx}$, where $\omega_n=\exp(2\pi i/n)$. Thus, by \eqref{evalue_cayley}, the corresponding discriminant eigenvalues are
	\[
	\mu_j=\frac1{|S|}\sum_{s\in S}\omega_n^{js}
	=\frac1{|S|}\sum_{s\in S}\cos\left(\frac{2\pi js}{n}\right)
	\quad\text{for } 0\leq j\leq n-1.
	\]
	The proof of Theorem~\ref{pgst_cayley} also gives the following equivalent condition for circulant graphs.
	\begin{coro}\label{pgst_circulant2}
		The circulant graph $\cay(\Zl_n,S)$ exhibits PGST between $u$ and $v$ if and only if $n$ is even with $v-u=n/2$ and there exists a sequence of positive integers $\{t_k\}$ such that $T_{t_k}(\mu_j)$ converges to $(-1)^j$ for $0\leq j\leq n-1$.
	\end{coro}
	
	\begin{proof}
		Suppose that $G$ exhibits PGST. By Corollary~\ref{pgst_circulant1}, $n$ is even and transfer occurs between $u$ and $u+n/2$. Since $\chi_j(n/2)=(-1)^j$, equation~\eqref{num} gives the required sequence.
		
		Conversely, suppose that $n$ is even and such a sequence exists. For any $u\in\Zl_n$, equation~\eqref{sd_cayley} gives
		\[
		\lim_{k \to \infty}\e_{u+n/2}^{\,t}T_{t_k}(P)\eu=\lim_{k \to \infty}\frac1n\sum_{j=0}^{n-1}(-1)^jT_{t_k}(\mu_j)=1.
		\]
		Hence Lemma~\ref{dpgst} proves PGST between $u$ and $u+n/2$.
	\end{proof}
	\subsection{Cubelike graphs}
	Let $G=\cay(\Fl_2^r,S)$, where $\emptyset\ne S\subseteq\Fl_2^r\setminus\{0\}$. The characters are $\chi_a(x)=(-1)^{a\cdot x}$, where $a\in\Fl_2^r$. By \eqref{evalue_cayley}, the corresponding discriminant eigenvalues are
	\[
	\mu_a=\frac1d\sum_{s\in S}(-1)^{a\cdot s}
	=\frac{d-2j_a}{d},
	\]
	where $d=|S|$ and $j_a=|\{s\in S:a\cdot s=1\}|$.
	\begin{lemma}\label{pgst_cubelike2}
		The cubelike graph $\cay(\Fl_2^r,S)$ exhibits PGST between $u$ and $v$ if and only if there exists a sequence of positive integers $\{t_k\}$ such that $T_{t_k}(\mu_a)$ converges to $(-1)^{a\cdot(v-u)}$ for each $a\in\Fl_2^r$.
	\end{lemma}
	
	\begin{proof}
		Put $w=v-u$. By \eqref{sd_cayley}, for every positive integer $t$,
		\begin{equation}\label{crb1218}
			\evt T_t(P)\eu
			=\frac1{2^r}\sum_{a\in\Fl_2^r}
			(-1)^{a\cdot w}T_t(\mu_a).
		\end{equation}
		Each summand $(-1)^{a\cdot w}T_{t_k}(\mu_a)$ in \eqref{crb1218} belongs to $[-1,1]$. Therefore $\lim_{k \to \infty}\evt T_{t_k}(P)\eu=1$ if and only if $\lim_{k \to \infty}T_{t_k}(\mu_a)=(-1)^{a\cdot w}$ for each $a\in\Fl_2^r$. Thus the result follows from Lemma~\ref{dpgst}.
	\end{proof}

	Put $\sigma=\sum_{s\in S}s$, and suppose that $\sigma\ne0$. Since $a\cdot\sigma\equiv j_a\pmod2$, Lemma~\ref{pgst_cubelike2} shows that PGST occurs between $u$ and $u+\sigma$ if and only if there exists a sequence of positive integers $\{t_k\}$ such that $\lim_{k \to \infty}T_{t_k}(\mu_a)=(-1)^{j_a}$ for every $a\in\Fl_2^r$.
	
	Since $\mu_a=1-2j_a/d$ and $0\leq j_a\leq d$, it suffices to find a sequence satisfying
	\begin{equation}\label{cubelike_common_times}
		\lim_{k \to \infty}T_{t_k}(1-2j/d)=(-1)^j\quad \text{for}\quad0\leq j\leq d.
	\end{equation}
	Indeed, taking $j=j_a$ gives the required limit for each character. This condition depends only on $d$, so the same sequence works for every cubelike graph of degree $d$ with $\sigma\ne0$.

	We use the valuations $v_2$ and $v_{\mathfrak q}$ defined in subsection~\ref{valuation}, with the convention $v_{\mathfrak q}(0)=\infty$.

	\begin{theorem}\label{cont_cubelike}
		Let $G=\cay(\Fl_2^r,S)$, and put $d=|S|$ and $\sigma=\sum_{s\in S}s$. If $\sigma\ne0$ and $v_2(d)\in\{0,1,3\}$, then $G$ exhibits PGST between $u$ and $u+\sigma$ for each $u\in\Fl_2^r$.
	\end{theorem}
	
	\begin{proof}
		It suffices to find positive integers $t_k$ satisfying \eqref{cubelike_common_times}. Put $\theta_j=\arccos(1-2j/d)$ and $z_j=e^{i\theta_j}$ for $0\leq j\leq d$.
		
		We use Theorem~\ref{kronecker} with $\alpha_j=\theta_j/(2\pi)$ and $\beta_j=j/2$. Its condition~(ii) is equivalent to
		\begin{equation}\label{crb_cubelike_parity}
			\text{if }\prod_{j=0}^d z_j^{\ell_j}=1,\quad\text{then }
			\sum_{j=0}^d j\ell_j\equiv0\pmod2\quad\text{for }\ell_j\in\Zl.
		\end{equation}
		Indeed, the product equals $1$ exactly when $\sum_j\ell_j\alpha_j\in\mathbb Z$, while $\sum_j\ell_j\beta_j=\frac12\sum_jj\ell_j$.
		
		If \eqref{crb_cubelike_parity} holds, Theorem~\ref{kronecker} gives integers $\tau_k$ and $p_{j,k}$ such that
		\[
		\lim_{k \to \infty}(\tau_k\theta_j-2\pi p_{j,k})=j\pi
		\quad\text{for }0\leq j\leq d.
		\]
		Taking cosine gives $\cos(\tau_k\theta_j)\longrightarrow(-1)^j$. The limit for $j=1$ shows that $\tau_k\ne0$ for all sufficiently large $k$. After removing finitely many initial terms, set $ t_k = |\tau_k|$. Since cosine is even, these positive integers satisfy \eqref{cubelike_common_times}, giving the required PGST.
		
		It remains to prove \eqref{crb_cubelike_parity}. Suppose that $\prod_{j=0}^d z_j^{\ell_j}=1$, and put $\ell=\sum_{j=0}^d|\ell_j|$. Write $y_j=i\sqrt{j(d-j)}$. Then $y_j$ is an algebraic integer, $y_j^2=-j(d-j)$, and $z_j^{\pm1}=({d-2j\pm2y_j})/{d}$.
		
		First suppose that $v_2(d)=0$. Choose $\epsilon_j\in\{-1,1\}$ such that $\ell_j=\epsilon_j|\ell_j|$. Then it follows that
		\[
		d^\ell=\prod_{j=0}^d(d-2j+2\epsilon_jy_j)^{|\ell_j|}.
		\]
		Expand and take real parts. Terms with an odd number of factors $y_j$ disappear. Every remaining term involving the $y_j$ has at least two factors of $2$, and hence a factor of $4$. Thus
		\[
		d^\ell=\prod_{j=0}^d(d-2j)^{|\ell_j|}+4\eta
		\]
		for some algebraic integer $\eta$. This equality gives $\eta\in\mathbb Q$, so $\eta\in\mathbb Z$. Since $d-2j\equiv d(-1)^j\pmod4$, we obtain $d^\ell\equiv d^\ell(-1)^{\sum_{j=0}^d j|\ell_j|}\pmod4$. Since $d$ is odd, $\sum_{j=0}^d j|\ell_j|$ is even. Since $|\ell_j|\equiv\ell_j\pmod2$, this proves \eqref{crb_cubelike_parity}.
		
		Next suppose that $v_2(d)=1$. Let $d=2e$, where $e$ is odd. Squaring $z_j$ and $z_j^{-1}$, we obtain
		\[
		z_j^{\pm2}=\frac{p_j\pm2(e-j)y_j}{e^2},
		\quad\text{where }
		p_j=2(e-j)^2-e^2\equiv(-1)^j\pmod4.
		\]
		Apply the same real-part argument to $\prod_{j=0}^d(z_j^2)^{\ell_j}=1$. Since $e^2\equiv1\pmod4$, we obtain
		\[
		1\equiv e^{2\ell}
		\equiv\prod_{j=0}^d p_j^{|\ell_j|}
		\equiv(-1)^{\sum_{j=0}^d j|\ell_j|}\pmod4.
		\]
		Thus $\sum_{j=0}^d j\ell_j$ is even.
		
		Finally, suppose that $v_2(d)=3$. Let $d=8e$, where $e$ is odd. Let $K=\mathbb Q(z_0,\ldots,z_d)$, choose a prime ideal $\mathfrak q$ of $\mathcal O_K$ above $2$, and put $\kappa=e(\mathfrak q/2)$. Since
		\[
		z_j+z_j^{-1}=2-\frac{4j}{d}=\frac{4e-j}{2e},
		\]
		its valuation is $-\kappa$ when $j$ is odd and is nonnegative when $j$ is even.
		
		If $v_{\mathfrak q}(z_j)\ne0$, then $z_j$ and $z_j^{-1}$ have distinct valuations, giving $v_{\mathfrak q}(z_j+z_j^{-1})=-|v_{\mathfrak q}(z_j)|$. If $v_{\mathfrak q}(z_j)=0$, their sum has nonnegative valuation. Therefore, $v_{\mathfrak q}(z_j)\in\{-\kappa,\kappa\}$ for odd $j$, $v_{\mathfrak q}(z_j)=0$ for even $j$. Taking valuations in the product relation and dividing by $\kappa$ gives
		\[
		0=\sum_{\substack{0\leq j\leq d\\j\text{ odd}}}
		\ell_j\frac{v_{\mathfrak q}(z_j)}{\kappa}
		\equiv
		\sum_{\substack{0\leq j\leq d\\j\text{ odd}}}\ell_j
		\equiv\sum_{j=0}^d j\ell_j\pmod2.
		\]
		Thus \eqref{crb_cubelike_parity} holds in all three cases, and PGST follows from the argument above.
	\end{proof}
	For the graphs in Theorem~\ref{cont_cubelike}, if $d\geq3$, then $v_2(d)\in\{0,1,3\}$ implies $d\notin\{1,2,4\}$. Applying Theorem~\ref{complete_classification} to each 	connected component therefore rules out PST. Consequently, these graphs exhibit PGST without PST whenever $d\geq3$. For $d=1$ and $d=2$, the components are $K_2$ and $C_4$, respectively, and exhibits PST.

	In $Q_d$, the only vertices fixed by all coordinate permutations are $0$ and the all-ones vector $\jb$. Thus Lemma~\ref{aut_pgst2} implies that PGST from $0$ to a different vertex can occur only to $\jb$. By translation, PGST between distinct vertices can occur only between $u$ and $u+\jb$. Chan and Zhan~\cite[Theorem~7.7]{pgst1} proved PGST between $u$ and $u+\jb$ on prime-dimensional hypercubes. We now determine all dimensions for which the hypercube exhibits PGST.
	
	\begin{theorem}\label{cont_hypercube}
		Let $d\geq1$. The hypercube $Q_d$ exhibits PGST if and only if $v_2(d)\in\{0,1,3\}$. Equivalently, $d$ is odd, $d\equiv2\pmod4$, or $d\equiv8\pmod{16}$. 
	\end{theorem}
	
	\begin{proof}
		The connection set of $Q_d$ consists of the standard basis vectors of $\Fl_2^d$, whose sum is the all-ones vector $\jb\ne0$. Thus Theorem~\ref{cont_cubelike} gives PGST between $u$ and $u+\jb$ whenever $v_2(d)\in\{0,1,3\}$.
		
		Conversely, suppose that $Q_d$ exhibits PGST. By the preceding discussion and translation, PGST occurs between $0$ and $\jb$. For a vector $a$ with $j$ nonzero coordinates, we have $\mu_a=1-2j/d$ and $(-1)^{a\cdot\jb}=(-1)^j$. Since every $0\leq j\leq d$ occurs, Lemma~\ref{pgst_cubelike2} gives positive integers $t_k$ satisfying \eqref{cubelike_common_times}.
		We show that these limits are impossible when $v_2(d)=2$ or $v_2(d)\geq4$, using the Chebyshev identity $T_t(T_m(x))=T_m(T_t(x))$.
		
		First suppose that $v_2(d)=2$, and write $d=4e$ with $e$ odd. Taking $j=e$ and $j=3e$ in \eqref{cubelike_common_times} gives $\lim_{k \to \infty}T_{t_k}(1/2)=-1$ and $\lim_{k \to \infty}T_{t_k}(-1/2)=-1$. However, $T_2(1/2)=-1/2$ gives $\lim_{k \to \infty}T_{t_k}(-1/2) =\lim_{k \to \infty}T_2\bigl(T_{t_k}(1/2)\bigr)= T_2(-1)=1,$
		a contradiction.
		
		Now suppose that $v_2(d)=s\geq4$, and write $d=2^se$ with $e$ odd. Since $3d/8$ is an even integer and $1-2(3d/8)/d=1/4$, equation \eqref{cubelike_common_times} gives $\lim_{k \to \infty}T_{t_k}(1/4)=1$.
		
		The Chebyshev recurrence gives $T_m(1/4)=b_m/2^{m+1}$ for $m\geq1$, where $b_1=1$, $b_2=-7$, and $b_{m+1}=b_m-4b_{m-1}$ for $m\geq2$. Thus every $b_m$ is an odd integer. Taking $m=s-2$, we obtain $T_m(1/4)=1-2j/d$, where $j=e(2^{s-1}-b_m)$.
		Here $j$ is odd, and $0\leq j\leq d$ because $|T_m(1/4)|\leq1$. Therefore, \eqref{cubelike_common_times} requires $\lim_{k \to \infty}T_{t_k}(T_m(1/4))=-1$. However, $\lim_{k \to \infty}T_{t_k}\bigl(T_m(1/4)\bigr)	=\lim_{k \to \infty}T_m\bigl(T_{t_k}(1/4)\bigr)=T_m(1)=1$,  	again a contradiction. Hence PGST occurs only when $v_2(d)\in\{0,1,3\}$.
	\end{proof}
	Among the hypercubes exhibiting PGST in Theorem~\ref{cont_hypercube}, only $Q_1$ and $Q_2$ exhibit
	PST. Indeed, $Q_1=K_2$ and $Q_2=C_4$, while every remaining admissible dimension $d$ lies outside $\{1,2,4\}$. Thus Theorem~\ref{complete_classification} shows that $Q_d$ exhibits PGST without PST precisely when $d\geq3$ and $v_2(d)\in\{0,1,3\}$.
	
	\subsection{Prism}
	Let $m\geq2$ and $\Gamma=\Zl_m\times\Zl_2$. The graph $G=K_m\square K_2$ has vertex set $\Gamma$, where $(a,b)$ and $(c,d)$ are adjacent if either $a=c$ and $b\ne d$, or $a\ne c$ and $b=d$. Thus
	\[G=\cay(\Gamma,S)\quad\text{and}\quad S=((\Zl_m\setminus\{0\})\times\{0\})\cup\{(0,1)\}.\]
	Then $\langle S\rangle=\Gamma$ and $|S|=m$. The graph $G$ is known as the prism of a complete graph. For $(r,s)\in\Gamma$, the corresponding character is $\chi_{(r,s)}(a,b)=\omega_m^{ra}(-1)^{sb}$. By \eqref{evalue_cayley},
	\[
	\mu_{(r,s)}
	=\frac1m\left(\sum_{a=1}^{m-1}\omega_m^{ra}+(-1)^s\right).
	\]
	Therefore, we have
	\begin{equation}\label{prism_ind_spectrum}
		\mu_{(r,s)}=
		\begin{cases}
			1&\text{if }r=0 \text{ and }s=0,\\
			\frac{m-2}{m}&\text{if }r=0\text{ and }s=1,\\
			0&\text{if }r\ne0\text{ and } s=0,\\
			-\frac2m&\text{if }r\ne0\text{ and }s=1.
		\end{cases}
	\end{equation}
	Write $\alpha=\arccos((m-2)/m)$ and $\beta=\arccos(-2/m)$.
	
	\begin{theorem}\label{prism_independent}
		Let $m\geq2$. Then the prism $K_m\square K_2$ exhibits PGST if and only if $m\ne4$.
	\end{theorem}
	
	\begin{proof}
		First, let $m=2$. Then the graph  $K_2\square K_2$ is cycle on $4$ vertices, and hence by Lemma~\ref{pgst_cycle},  $K_2\square K_2$ exhibits PGST.
		
		Now let $m\geq3$. Every permutation $\pi$ of $\Zl_m$ induces the automorphism $(a,b)\mapsto(\pi(a),b)$. Such a permutation can fix one of the first coordinates and move any of the others. By Lemma~\ref{aut_pgst2}, PGST can therefore occur only between $(a,0)$ and $(a,1)$. By translations, it suffices to consider $u=(0,0)$ and $v=(0,1)$.
		
		Here $w=v-u=(0,1)$ has order $2$, so condition~(i) of Theorem~\ref{pgst_cayley} holds. Also, $\chi_{(r,s)}(w)=(-1)^s$, giving $\Gamma_{1,w}=\Zl_m\times\{1\}$. For integers $\{\ell_{(r,s)}:(r,s)\in\Gamma\}$, put
		\[h=\ell_{(0,1)},\quad k=\sum_{r=1}^{m-1}\ell_{(r,1)},\quad\text{and }~\ell=\sum_{r=1}^{m-1}\ell_{(r,0)}.\]
		Using \eqref{prism_ind_spectrum}, we obtain
		\[
		\begin{aligned}
			\sum_{(r,s)\in\Gamma}\ell_{(r,s)}\arccos\mu_{(r,s)}
			&=h\alpha+k\beta+\frac{\pi \ell}{2},\\
			\sum_{(r,s)\in\Gamma_{1,w}}\ell_{(r,s)}
			&=h+k.
		\end{aligned}
		\]
		Thus condition~(ii) of Theorem~\ref{pgst_cayley} requires $h+k$ to be even whenever $h\alpha+k\beta+\pi \ell/2\in2\pi\Zl$.
		
		If $m=4$, take $\ell_{(1,1)}=3$ and all other coefficients zero. Then $h=\ell=0$, $k=3$, and $\beta=2\pi/3$. The angle sum is $3\beta=2\pi$, but the sum over $\Gamma_{1,w}$ is $3$. Hence condition~(ii) of Theorem~\ref{pgst_cayley} fails, and $G$ does not exhibit PGST.
		
		Suppose now that $m\ne4$. Put $z_1=e^{i\alpha}$ and $z_2=e^{i\beta}$. Since $m/\gcd(m,4)>1$, choose a prime divisor $q$ of this integer. Let $K=\Ql(z_1,z_2,i)$, choose a prime ideal $\mathfrak q$ of $\mathcal O_K$ above $q$, and put $c=-v_{\mathfrak q}(4/m)>0$. We have
		\[
		z_1+z_1^{-1}=2-\frac4m,
		\qquad
		z_2+z_2^{-1}=-\frac4m.
		\]
		Since $v_{\mathfrak q}(2)\geq0$, both right-hand sides have valuation $-c$. For $j=1,2$, this forces $v_{\mathfrak q}(z_j)\ne0$, since otherwise the sum would have nonnegative valuation. The two summands then have distinct valuations, giving $v_{\mathfrak q}(z_j+z_j^{-1})=-|v_{\mathfrak q}(z_j)|=-c$. Thus $v_{\mathfrak q}(z_1)\in\{c,-c\}$ and $v_{\mathfrak q}(z_2)\in\{c,-c\}$.
		
		Now suppose that $h\alpha+k\beta+\pi \ell/2\in2\pi\Zl$. Exponentiating gives $z_1^h z_2^k i^\ell=1$. Since $i^4=1$, we have $v_{\mathfrak q}(i)=0$. Taking valuations therefore yields $hv_{\mathfrak q}(z_1)+kv_{\mathfrak q}(z_2)=0$. As both valuations belong to $\{c,-c\}$, we have $h=k$ or $h=-k$. Consequently,
		\[
		\sum_{(r,s)\in\Gamma_{1,w}}\ell_{(r,s)}
		=h+k\equiv0\pmod2.
		\]
		Thus condition~(ii) of Theorem~\ref{pgst_cayley} holds, and hence PGST occurs between $(0,0)$ and $(0,1)$.
	\end{proof}
	The prism $K_m\square K_2$ is an abelian Cayley graph of valency $m$. Therefore, by Theorem~\ref{complete_classification}, it does not exhibit PST when $m=3$ or $m\geq5$. Combining this observation with Theorem~\ref{prism_independent}, we conclude that $K_m\square K_2$ exhibits PGST without PST precisely when $m\geq3$ and $m\ne4$. The case $m=2$ is $C_4$, which exhibits PST, whereas $m=4$ exhibits neither PST nor PGST.
	\subsection{Unitary Cayley graphs}
	
	In this section, we give a complete classification of PGST on unitary Cayley graphs. The \emph{unitary Cayley graph} on $\Zl_n$ is $G_{\Zl_n}=\cay(\Zl_n,\Zl_n^\times)$, where $\Zl_n^\times$ is the group of units modulo $n$. See~\cite{ucg} for more details.
	
	We write $\mu_{\rm M}$ for the M\"obius function to distinguish it from the discriminant eigenvalues $\mu_j$. For $j\in\Zl_n$, the Ramanujan sum is
	\begin{equation}\label{ramanujan1}
		c(j,n)=\sum_{k\in\Zl_n^\times}\omega_n^{jk}=\sum_{r\mid\gcd(j,n)}r\,\mu_{\rm M}(n/r).
	\end{equation}
	By \eqref{evalue_cayley}, the corresponding discriminant eigenvalue is
	\begin{equation}\label{ramanujan2}
		\mu_j=\frac{c(j,n)}{\varphi(n)}=\frac{\mu_{\rm M}(t_{n,j})}{\varphi(t_{n,j})},\quad\text{where }t_{n,j}=\frac{n}{\gcd(n,j)}.	
	\end{equation}
	We use the convention $\gcd(n,0)=n$. The formulas for $c(j,n)$ can be found in~\cite{ramanujan}.

	\begin{lemma}\label{pgst2_circ}
		Let  $a$  be odd and $b$ be even with $0\leq a,b\leq n-1$. If $\mu_a=0=\mu_b$ in the circulant graph $\cay(\Zl_n,S)$, then the graph does not exhibit PGST.
	\end{lemma}
	\begin{proof}
		Let $\ell_a=1$, $\ell_b=3$, and $\ell_r=0$ for $r\in\Zl_n\setminus\{a,b\}$. Then it is easy to verify that 
		\[\sum_{r=0}^{n-1}\ell_r\arccos\mu_r\equiv0\pmod{2\pi}.\]
		However,
		\[\sum_{\substack{r=0\\ r\text{ odd}}}^{n-1}\ell_r=1 \not\equiv0\pmod 2.\]
		Therefore by Corollary~\ref{pgst_circulant1}, $\cay(\Zl_n,S)$ does not exhibit PGST.
	\end{proof}
	Using the preceding result, we show that if $G_{\Zl_n}$ exhibits PGST, then $n$ is squarefree of odd primes and  $2$ is a factor of $n$ with exponent at most $2$.
	
	\begin{lemma}\label{pgst_uc1}
		If $G_{\Zl_n}$ exhibits PGST, then $n=2m$ or $n=4m$ for some odd squarefree positive integer $m$.
	\end{lemma}
	
	\begin{proof}
		Suppose that $G_{\Zl_n}$ exhibits PGST. By Corollary~\ref{pgst_circulant1}, $n$ is even. Write $n=2^s m$, where
		$s\geq 1$ and $m$ is odd.
		
		Suppose that $p^2\mid m$ for some odd prime $p$, and let $k=n/p^2$. Then $k$ is even. By \eqref{ramanujan2}, $\mu_1=\mu_k=0$, which contradicts Lemma~\ref{pgst2_circ}. Hence $m$ is squarefree.
		
		Now suppose that $s\geq 3$. By \eqref{ramanujan2}, $\mu_1=\mu_{2m}=0$, again contradicts Lemma~\ref{pgst2_circ}. Thus $s=1$ or $s=2$, giving $n=2m$ or $n=4m$.
	\end{proof} 
	
	The case $n=2m$ requires an additional arithmetic condition, whereas every graph of order $4m$, with $m$ odd and squarefree, exhibits PGST.

	\begin{theorem}\label{pgst_unitary}
		The unitary Cayley graph $G_{\Zl_n}$ exhibits PGST precisely if and only if one of the following holds.
		\begin{enumerate}[label=(\roman*)]
			\item $n=4m$ for some odd squarefree positive integer $m$.
			\item $n=2m$ for some odd squarefree positive integer $m$ such that, for each pair of positive divisors $d$ and $e$ of $m$, if $\varphi(d)=\varphi(e)$ then $\mu_{\rm M}(d)=\mu_{\rm M}(e)$.
		\end{enumerate}
	\end{theorem}
	
	\begin{proof}
		By Lemma~\ref{pgst_uc1},  $n=2m$ or $n=4m$ for some odd squarefree positive integer $m$. Let $d_j=m/\gcd(m,j)$ and $\mathcal D=\{\varphi(d):d\mid m,\ \varphi(d)\geq4\}$. Note that $\mathcal D=\emptyset$ exactly when $m\in\{1,3\}$.

		First let $n=4m$. Then the value of  $t_{4m,j}$ is $4d_j$, $2d_j$, or $d_j$ according as $j$ is odd, $j\equiv2\pmod4$, or $j\equiv0\pmod4$. Since $\mu_{\rm M}(4d_j)=0$, \eqref{ramanujan2} gives
		\[
		\mu_j=
		\begin{cases}
			0&\text{if $j$ odd},\\
			-\frac{\mu_{\rm M}(d_j)}{\varphi(d_j)}&\text{if }j\equiv2\pmod4,\\
			\frac{\mu_{\rm M}(d_j)}{\varphi(d_j)}&\text{if }j\equiv0\pmod4.
		\end{cases}
		\]
		
		For $\mathcal D\ne\emptyset$, apply Corollary~\ref{progression_density} with $\ell=12$, $b=6$, and $\beta_a=0$ for each $a\in\mathcal D$. This gives positive integers $t_k\equiv6\pmod{12}$ such that $e^{it_k\arccos(1/a)}\rightarrow1$. Taking real parts and using the evenness of $t_k$, we obtain $T_{t_k}(\pm1/a)\rightarrow1$ for each $a\in\mathcal D$. If $\mathcal D=\emptyset$, then each eigenvalue belongs to $\{0,\pm1,\pm1/2\}$, and so take $t_k=12k+6$. 
		
		For the times $t_k\equiv6\pmod{12}$, we obtain $T_{t_k}(0)=-1$, $T_{t_k}(\pm1)=1$, and $T_{t_k}(\pm1/2)=\cos(t_k\pi/3)=1$.
		For odd $j$, $\mu_j=0$, and for even $j$, $\mu_j$ has the form $\pm1/\varphi(d)$ for some $d\mid m$. Consequently, $T_{t_k}(\mu_j)\rightarrow(-1)^j$ for each $j$. Therefore by Corollary~\ref{pgst_circulant2}, PGST occurs between $0$ and $n/2$.

		Now let $n=2m$. Since $m$ is odd, $t_{2m,j}=d_j$ for even $j$ and $t_{2m,j}=2d_j$ for odd $j$. Using $\mu_{\rm M}(2d)=-\mu_{\rm M}(d)$ and $\varphi(2d)=\varphi(d)$ for odd $d$, equation~\eqref{ramanujan2} gives
		\begin{equation}\label{uc_two_parities}
			\mu_j=
			\begin{cases}
				\frac{\mu_{\rm M}(d_j)}{\varphi(d_j)}&\text{if $j$ even},\\ 
				-\frac{\mu_{\rm M}(d_j)}{\varphi(d_j)}&\text{if $j$ odd}.
			\end{cases}
		\end{equation}
		
		Suppose there exist divisors $d,e$ of $m$ satisfy $\varphi(d)=\varphi(e)$ and $\mu_{\rm M}(d)\ne\mu_{\rm M}(e)$. Since $m$ is squarefree, both $d$ and $e$ are squarefree, and hence $\mu_{\rm M}(d)=-\mu_{\rm M}(e)$.
		
		Take $a=2m/d$ and $b=m/e$. Then $a$ is even, and $b$ is odd. Since $t_{2m,a}=d$ and $t_{2m,b}=2e$, we obtain
		\[
		\mu_a=\frac{\mu_{\rm M}(d)}{\varphi(d)}
		=-\frac{\mu_{\rm M}(e)}{\varphi(e)}
		=\mu_b,
		\]
		which contradicts Corollary~\ref{pgst_circulant2}. Therefore if $\varphi(d)=\varphi(e)$ then $\mu_{\rm M}(d)=\mu_{\rm M}(e)$.

		For each $a=\varphi(d)$ with $d\mid m$, define $\kappa_a=\mu_{\rm M}(d)$. For $\mathcal D\ne\emptyset$, we apply Corollary~\ref{progression_density} with $\ell=6$, $b=3$, and $\beta_a=(1-\kappa_a)\pi/2$. Taking real parts gives a common sequence of positive integers $t_k\equiv3\pmod6$ such that $T_{t_k}(1/a)\rightarrow\kappa_a$ for each $a\in\mathcal D$. If $\mathcal D=\emptyset$, take $t_k=6k+3$.
		
		For odd $d$, we have $\varphi(d)<4$ if and only if $d\in\{1,3\}$. Also, $T_{t_k}(1)=1=\kappa_1$ and, if $3\mid m$, $T_{t_k}(1/2)=\cos(t_k\pi/3)=-1=\kappa_2$. Combining these identities with the preceding limits gives
		\[
		T_{t_k}\left(\frac1{\varphi(d)}\right)
		\longrightarrow\kappa_{\varphi(d)}=\mu_{\rm M}(d)
		\quad \text{for }d\mid m.
		\]	
		Since $t_k$ is odd, $T_{t_k}(-x)=-T_{t_k}(x)$. Thus for $a=\varphi(d_j)$, \eqref{uc_two_parities} gives $T_{t_k}(\mu_j)=(-1)^j\kappa_a T_{t_k}(1/a)\rightarrow(-1)^j$.
		Therefore by Corollary~\ref{pgst_circulant2}, $G_{\Zl_{2m}}$ exhibits PGST.
	\end{proof}
	
	Theorem~\ref{complete_classification} shows that $G_{\Zl_n}$ exhibits PST if and only if 	$n\in\{2,4,6,12\}$. Consequently, every graph in Theorem~\ref{pgst_unitary} with $n\notin\{2,4,6,12\}$ exhibits PGST without PST.
	
	\begin{ex}\label{example_546}
		Let $m=3\cdot7\cdot13=273$. Then $\varphi(13)=\varphi(21)=12$, but $\mu_{\rm M}(13)=-1$,  $\mu_{\rm M}(21)=1$.
		Thus $G_{\Zl_{546}}$ does not exhibit PGST. On the other hand, $G_{\Zl_{1092}}$ exhibits PGST.
	\end{ex}
	
	Recall that $v_2(a)$ denotes the exponent of $2$ in a positive integer $a$.
	
	\begin{coro}\label{unitary_families}
		Let $m$ be an odd squarefree positive integer. Then the following hold.
		\begin{enumerate}[label=(\roman*)]
			\item If $m$ has at most two prime divisors, then $G_{\Zl_{2m}}$ exhibits PGST.
			\item If there exists an integer $c\geq1$ such that $v_2(p-1)=c$ for each prime $p\mid m$, then $G_{\Zl_{2m}}$ exhibits PGST. In particular, this holds when each prime divisor of $m$ is congruent to $3$ modulo $4$.
			\item If $273\mid m$, then $G_{\Zl_{2m}}$ does not exhibit PGST.
		\end{enumerate}
	\end{coro}
	\begin{proof}
		By Theorem~\ref{pgst_unitary}, it suffices to determine whether $\varphi(d)=\varphi(e)$ implies $\mu_{\rm M}(d)=\mu_{\rm M}(e)$ for all divisors $d$ and $e$ of $m$.
		\begin{enumerate}[label=(\roman*)]
			\item If $m=1$ or $m=p$ is prime, the totients of its divisors are distinct. Now let $m=pq$, where $p<q$ are odd primes. Its divisors are $1,p,q,pq$, and $1<p-1<q-1<(p-1)(q-1)$. Thus, in each case, $\varphi(d)=\varphi(e)$ implies $d=e$, and hence $\mu_{\rm M}(d)=\mu_{\rm M}(e)$. Therefore $G_{\Zl_{2m}}$ exhibits PGST.
			\item  Let $d,e\mid m$, and let $r$ and $s$ denote the numbers of distinct prime divisors of $d$ and $e$, respectively. Since $d$ and $e$ are squarefree,
			\[
			v_2(\varphi(d))=\sum_{p\mid d}v_2(p-1)=cr\quad\text{and}\quad v_2(\varphi(e))=cs.
			\]
			If $\varphi(d)=\varphi(e)$, then $cr=cs$. Since $c\geq1$, we have $r=s$, and therefore $\mu_{\rm M}(d)=(-1)^r=(-1)^s=\mu_{\rm M}(e)$. Again, Theorem~\ref{pgst_unitary} proves PGST. If each prime divisor $p$ of $m$ satisfies $p\equiv3\pmod4$, then $v_2(p-1)=1$, so the conclusion follows by taking $c=1$.
			
			\item  Note that $273=3\cdot7\cdot13$. Thus $13$ and $21$ divide $m$, and $\varphi(13)=12=\varphi(21)$, $\mu_{\rm M}(13)=-1\ne1=\mu_{\rm M}(21)$. Hence, the condition of Theorem~\ref{pgst_unitary}(ii) is not satisfied. Thus $G_{\Zl_{2m}}$ does not exhibit PGST.\qedhere
		\end{enumerate}
	\end{proof}

	\section*{Acknowledgments}
	The first author  thanks the Prime Minister’s Research Fellowship, Government of India, for financial support (PMRF-ID: 1903298).


\begin{thebibliography}{10}
		\bibitem{random}
		D.~Aharonov, L.~Davidvich, and N.~Zagury.
		\newblock Quantum random walks.
		\newblock {\em Physical Review A}. 48:1680--1690, 1993.
		
		\bibitem{algorithm}
		A.~Ambainis.
		\newblock Quantum walk algorithm for element distinctness.
		\newblock {\em SIAM Journal on Computing}. 37(1):210--239, 2007.
		
		\bibitem{bhakta_abelian}
		K. Bhakta and B. Bhattacharjya.
		\newblock A complete classification of perfect state transfer in {G}rover walks on abelian {C}ayley graphs.
		\newblock {\em arXiv:2610.21946v1}, 2026. 
		
		
		\bibitem{bhakta4}
		K.~Bhakta and B.~Bhattacharjya.
		\newblock Perfect state transfer in Grover walks on association schemes and distance-regular graphs.
		\newblock {\em arXiv:2405.01020v1}, 2025.
		
		\bibitem{bhakta3}
		K.~Bhakta and B.~Bhattacharjya.
		\newblock State transfer in Grover walks on unitary and quadratic unitary Cayley graphs over finite commutative rings.
		\newblock {\em Discrete Mathematics}. 349(8):115151, 2026.
		
		\bibitem{bmp_blowups}
		B.~Bhattacharjya, H.~Monterde, and H.~Pal.
		\newblock Quantum walks on blow-up graphs.
		\newblock {\em Journal of Physics A: Mathematical and Theoretical}.
		57(33):335303, 2024.
		
		\bibitem{pgst_paths}
		C.M. van Bommel.
		\newblock A complete characterization of pretty good state transfer on paths.
		\newblock {\em Quantum Information and Computation}. 19(7-8):601--608, 2019.
		
		\bibitem{pgst_dihedral}
		X. Cao, D. Wang, and K. Feng.
		\newblock Pretty good state transfer on {C}ayley graphs over dihedral groups.
		\newblock {\em Discrete Mathematics}. 343(1):111636, 12, 2020.
		
		\bibitem{pgst1}
		A.~Chan and H.~Zhan.
		\newblock Pretty good state transfer in discrete-time quantum walks.
		\newblock {\em Journal of Physics A: Mathematical and Theoretical}. 56:165305, 2023.
		
		\bibitem{chen}
		Q.~Chen, C.~Godsil, M.~Sobchuk, and H.~Zhan.
		\newblock Hamiltonians of bipartite walks.
		\newblock {\em The Electronic Journal of Combinatorics}. 31:\#P4.10, 2024.
		
		
		\bibitem{information}
		M.~Christandl, N.~Datta, A.~Ekert, and A.J.~Landahl.
		\newblock Perfect state transfer in quantum spin networks.
		\newblock {\em Physical Review Letters}. 92(18):187902, 2004.
		
		
		\bibitem{conway_angles}
		J.H.~Conway, C.~Radin, and L.~Sadun.
		\newblock On angles whose squared trigonometric functions are rational.
		\newblock {\em  Discrete $\&$ Computational Geometry}. 22:321--332, 1999.
		
		\bibitem{pgst_double_star}
		X. Fan and C. Godsil.
		\newblock Pretty good state transfer on double stars.
		\newblock {\em Linear Algebra and its Applications}. 438(5):2346--2358, 2013.
		
		\bibitem{state}
		C.~Godsil.
		\newblock State transfer on graphs.
		\newblock {\em Discrete Mathematics}. 312(1):129--147, 2012.
		
		
		\bibitem{godsildct}
		C.~Godsil and H.~Zhan.
		\newblock Discrete-time quantum walks and graph structures.
		\newblock {\em Journal of Combinatorial Theory, Series A}. 167:181--212, 2019.
		
		
		\bibitem{kronecker}
		S.M.~Gonek and H.L.~Montgomery.
		\newblock Kronecker's approximation theorem.
		\newblock {\em Indagationes Mathematicae}. 27(2):506--523, 2016.
		
		\bibitem{guo_sch}
		K. Guo and V. Schmeits.
		\newblock State {T}ransfer in {D}iscrete-{T}ime {Q}uantum {W}alks via
		{P}rojected {T}ransition {M}atrices.
		\newblock {\em The Electronic Journal of Combinatorics}. 33(4):\#P4.4, 2026. 
		
		\bibitem{spec}
		Y.~Higuchi, N.~Konno, I.~Sato, and E.~Segawa.
		\newblock Spectral and asymptotic properties of Grover walks on crystal lattices.
		\newblock {\em Journal of Functional Analysis}. 267:4197–4235, 2014.
		
		
		\bibitem{ito}
		N.~Ito, T.~Matsuyama, and T.~Tsurii.
		\newblock Periodicity of Grover walks on complete graphs with self-loops.
		\newblock {\em Linear Algebra and its Applications}. 599:121--132, 2020.
		
		
		\bibitem{ucg}
		W.~Klotz and T.~Sander.
		\newblock Some properties of unitary Cayley graphs.
		\newblock {\em The Electronic Journal of Combinatorics}. 14:\#R45, 2007.
		
		\bibitem{kubota1}
		S.~Kubota and E.~Segawa.
		\newblock Perfect state transfer in Grover walks between states associated to vertices of a graph.
		\newblock {\em Linear Algebra and its Applications}. 646:238--251, 2022.
		
		\bibitem{mixed1}
		S.~Kubota, E.~Segawa, and T.~Taniguchi.
		\newblock Quantum walks defined by digraphs and generalized Hermitian adjacency matrices.
		\newblock {\em Quantum Information Processing}. 20:95, 2021.
		
		\bibitem{mixedpaths}
		S.~Kubota, H.~Seikdo, and H.~Yata.
		\newblock Periodicity of quantum walks defined by mixed paths and mixed cycles.
		\newblock {\em Linear Algebra and its Applications}. 630:15--38, 2021.
		
		
		\bibitem{regular}
		S.~Kubota, H.~Sekido, and H.~Yoshino.
		\newblock Regular graphs to induce even periodic Grover walks.
		\newblock {\em Discrete Mathematics}. 348(3):114345, 2025.	
		
		\bibitem{kubota2}
		S.~Kubota and K.~Yoshino.
		\newblock Circulant graphs with valency up to $4$ that admit perfect state transfer in Grover walks.
		\newblock {\em Journal of Combinatorial Theory,	Series A}. 216:106064, 2025.
		
		\bibitem{ouaknine_density}
		J. Ouaknine, J. Sousa-Pinto and J. Worrell,
		\newblock On the polytope escape problem for continuous linear dynamical systems.
		\newblock {\em Proceedings of the 20th International Conference on
			Hybrid Systems: Computation and Control}, ACM, 2017, pp.~11--17. 
		
		\bibitem{pal_pgst2}
		H.~Pal.
		\newblock More circulant graphs exhibiting pretty good state transfer.
		\newblock {\em Discrete Mathematics}. 341(4):889--895, 2018.
		
		
		
		\bibitem{ramanujan}
		S.~Ramanujan.
		\newblock On certain trigonometrical sums and their applications in the theory of numbers.
		\newblock {\em Transactions of the Cambridge Philosophical Society}. 22(13):259--276, 1918.
		
		
		\bibitem{rep}
		B.~Steinberg.
		\newblock Representation Theory of Finite Groups: An Introductory Approach.
		\newblock {\em Springer, New York}. 2009.
		
		
		
		\bibitem{ant_stewart}
		I.~Stewart and D. Tall.
		\newblock Algebraic Number Theory and Fermat's Last Theorem.
		\newblock {\em Chapman and Hall/CRC, New York}. 2001.
		
		
		\bibitem{yabuta}
		M.~Yabuta.
		\newblock A simple proof of Carmichael's theorem on primitive divisors.
		\newblock {\em The Fibonacci Quarterly}. 39(5):439--443, 2001.
		
		\bibitem{odd-grover}
		Y.~Yoshie.
		\newblock Odd-periodic Grover walks.
		\newblock {\em Quantum Information Processing}. 22:316, 2023.
		
		\bibitem{distance}
		Y.~Yoshie.
		\newblock Periodicity of Grover walks on distance-regular graphs.
		\newblock {\em Graphs and Combinatorics}. 35:1305--1321, 2019.
		
		\bibitem{zhan1}
		H.~Zhan.
		\newblock An infinite family of circulant graphs with perfect state transfer in discrete quantum walks.
		\newblock {\em Quantum Information Processing}. 18(12):1--26, 2019.
		
		\bibitem{pgst2}
		H. Zhan.
		\newblock Simple quantum coins enable pretty good state transfer on every hypercube.
		\newblock {\em Philosophical Transactions of the Royal Society A}. 383(2306):20240418, 12,
		2025.
		
		
		
		
		
		
	\end{thebibliography}
\end{document}